\documentclass[11pt, draft]{amsart}
\usepackage{amssymb, amstext, amscd, amsmath, amssymb}
\usepackage{mathtools, xypic, paralist, color, dsfont, rotating}
\usepackage{MnSymbol}
\usepackage{verbatim}
\usepackage{enumerate}
\usepackage{enumitem}
\usepackage{calc}
\usepackage{float}

\numberwithin{equation}{section}

\usepackage[a4paper]{geometry}
\usepackage{quoting}
\quotingsetup{vskip=.1in}
\quotingsetup{leftmargin=.17in}
\quotingsetup{rightmargin=.17in}

\makeatletter
\def\@cite#1#2{{\m@th\upshape\bfseries%
[{#1\if@tempswa{\m@th\upshape\mdseries, #2}\fi}]}}
\makeatother
\theoremstyle{plain}
\newtheorem{theorem}{Theorem}[section]

\newtheorem{proposition}[theorem]{Proposition}

\theoremstyle{definition}
\newtheorem{definition}[theorem]{Definition}
\newtheorem{example}[theorem]{Example}

\newtheorem{remark}[theorem]{Remark}

\theoremstyle{remark}

\mathtoolsset{centercolon}
  \newcommand{\B}{{\mathcal{B}}}

  \newcommand{\N}{{\mathcal{N}}}
\renewcommand{\O}{{\mathcal{O}}}

\renewcommand{\S}{{\mathcal{S}}}
  \newcommand{\T}{{\mathcal{T}}}

\def\al{\alpha}
\def\be{\beta}
\def\ga{\gamma}
\def\de{\delta}

\def\la{\lambda}
\def\La{\Lambda}

\def\si{\sigma}

\newcommand\vphi{\varphi}

\newcommand{\bC}{\mathbb{C}}

\newcommand{\bN}{\mathbb{N}}
\newcommand{\bT}{\mathbb{T}}
\newcommand{\bZ}{\mathbb{Z}}
\newcommand{\bR}{\mathbb{R}}

\newcommand{\fJ}{{\mathfrak{J}}}

\newcommand{\fS}{{\mathfrak{S}}}

\newcommand{\Bi}{{\mathbf{i}}}

\newcommand{\FOR}{\text{ for }}

\newcommand{\foral}{\text{ for all }}
\newcommand{\qand}{\quad\text{and}\quad}
\newcommand{\qif}{\quad\text{if}\quad}
\newcommand{\qiff}{\quad\text{if and only if}\quad}
\newcommand{\qfor}{\quad\text{for}\; }

\newcommand{\ca}{\mathrm{C}^*}

\newcommand{\wt}{\widetilde}

\newcommand{\Aut}{\operatorname{Aut}}

\newcommand{\Avt}{\operatorname{AVT}}

\newcommand{\Eq}{\operatorname{E}}
\newcommand{\fty}{\operatorname{fin}}

\newcommand{\GEq}{\operatorname{G-E}}

\newcommand{\mt}{\emptyset}

\newcommand{\supp}{\operatorname{supp}}
\newcommand{\Tr}{\operatorname{T}}

\newcommand{\sca}[1]{\langle#1\rangle} 
\newcommand{\nor}[1]{\Vert #1\Vert} 
\newcommand{\bo}[1]{\mathbf{#1}} 
\newcommand{\un}[1]{{\underline{#1}}} 
\newcommand{\umu}{\underline\mu}

\addtocontents{toc}{\protect\setcounter{tocdepth}{1}}

\begin{document}

\title[Equilibrium states and entropy theory for higher-rank graphs]{Applications of entropy of product systems: higher-rank graphs}

\author[E.T.A. Kakariadis]{Evgenios T.A. Kakariadis}
\address{School of Mathematics, Statistics and Physics\\ Newcastle University\\ Newcastle upon Tyne\\ NE1 7RU\\ UK}
\email{evgenios.kakariadis@ncl.ac.uk}



\begin{abstract}
We consider C*-algebras of finite higher-rank graphs along with their rotational action.
We show how the entropy theory of product systems with finite frames applies to identify the phase transitions of the dynamics.
We compute the positive inverse temperatures where symmetry breaks, and in particular we identify the subharmonic parts of the gauge-invariant equilibrium states.
Our analysis applies to positively weighted rotational actions through a recalibration of the entropies.
\end{abstract}

\maketitle

\section{Introduction}

In analogy to ideal gases, an action of $\bR$ on a C*-algebra induces a notion of equilibrium states that satisfy a KMS-condition at an inverse temperature $\be > 0$ \cite{BR97}.
Phase transitions of the KMS-simplex then provide one-dimensional measurements that can be used as obstructions for $\bR$-equivariant $*$-isomorphisms.
In the past two decades there has been an excessive amount of research for computing those and connecting them to the geometry of the structure the dynamics quantize.
Two important values $\be_c' \leq \be_c$ have been detected in this respect.
When $\be < \be_c'$ then the KMS${}_\be$-simplex is empty, and when $\be > \be_c$ then the KMS${}_\be$-simplex is induced by an ideal of (generalized) compacts.
The latter are said to be of finite type.
Thinking of the compacts as a representation of the classical world, the critical value $\be_c$ can be seen as an analogue of the critical mass for the double-slit experiment.

The work of Laca-Neshveyev \cite{LN04} and Laca-Raeburn \cite{LR10} has been central in this respect.
They have provided general tools for understanding the KMS-simplex of C*-algebras that are generated through a Fock-space construction.
In this paper we consider the Toeplitz-Cuntz-Krieger algebra $\N\T(\La)$ of a higher-rank graph $\La$, and its Cuntz-Krieger quotient $\N\O(\La)$.
In a series of papers an Huef-Laca-Raeburn-Sims \cite{HLRS14, HLRS15} have studied strongly connected higher-rank graphs and provided a connection of a KMS-state with its restriction to the diagonal.
Their work exhibits a link between phase transitions and the C*-structure: in \cite{HLRS15} it is shown that $\N\O(\La)$ admits a unique KMS-state if and only if it is simple, answering a question of Yang \cite{Yan17}.
Moving to more general cases an Huef-Kang-Raeburn \cite{HKR17} and Fletcher-an Huef-Raeburn \cite{FHR18} produced an algorithm to compute the KMS-simplices at $\be_c$ (and then reducing to other phase transitions) as long as the higher-rank graph admits a connectedness between irreducible components of different directions.
Christensen \cite{Chr18} removed all these conditions and remarkably tackled the general case.
He proved that every KMS${}_\be$-state admits a unique convex decomposition in subharmonic parts, one for each subset of directions.

Around the same time, the author \cite{Kak18} independently showed that this holds for a wider family of C*-algebras coming from product systems.
Building on \cite{Chr18, FHR18, HKR17, HLRS14, HLRS15, LN04} and under the existence of finite frames, we produced a parametrization of the gauge-invariant subharmonic simplices from principal information of the diagonal, rather than from their restriction to the diagonal.
For finite higher-rank graphs this produces a weak*-continuous bijection with appropriate unimodular vectors.
Afsar-Larsen-Neshveyev \cite{ALN18} extended the parametrization of the finite-type simplex for rather general product systems (over other quasi-lattices than $\bZ_+^N$). 

The tool introduced in \cite{Kak17,Kak18} uses several notions of entropy on finite frames.
As a consequence $\be_c'$ is completely identified as the nonnegative greatest lower bound of the tracial entropies, while $\be_c$ can be at most equal to the strong entropy of the finite frames.
This naturally leads to asking the following questions:
\begin{enumerate}
\item Is the strong entropy a phase transition for the dynamics on product systems?
\item How does the gauge-invariant KMS${}_\be$-simplex change while $\be$ varies?
\end{enumerate}

In the current paper we answer these questions to complete the study of phase transitions for finite higher-rank graphs initiated by an Huef-Laca-Raeburn-Sims \cite{HLRS14, HLRS15} and then examined by an Huef-Kang-Raeburn \cite{HKR17}, Fletcher-an Huef-Raeburn \cite{FHR18} and Christensen \cite{Chr18}; the rank-1 case is studied in \cite{HLRS15b} and \cite{Kak17}.
The main feature is that $\La$ is a set of commuting matrices with prescribed commutation relations between paths of different colours; we write
\begin{equation}
\La =(\La^{(1)}, \dots, \La^{(N)}; \sim)
\end{equation}
in this respect.
First of all we settle that $\be_c$ for $\N\T(\La)$ is exactly the entropy of the higher-rank graph, i.e.,
\begin{equation}
\be_c = \log \rho(\La) := \max \big\{ \log \rho(\La^{(1)}), \dots, \log \rho(\La^{(N)}) \big\}, 
\end{equation}
where $\rho(\La^{(i)})$ is the Perron-Frobenius eigenvalue of $\La^{(i)}$.
In Theorem \ref{T:main gi} we show that nonnegative phase transitions occur exactly at $\log \la$ for every 
\begin{equation}
1 \leq \la \in \fS := \{\rho(H) \mid \textup{ $H$ is a sink subgraph of $\La$} \},
\end{equation}
in the sense that for every such $\la$ there exists an $F \subsetneq \{1, \dots, N\}$ so that the $F$-subharmonic part $\Eq_{\log \la}^F(\N\T(\La))$ is non-empty, and that if $\be \notin \{\log \la \mid 1 \leq \la \in \fS\}$ then the only part that survives is the finite part, i.e., 
\begin{equation}
\Eq_\be(\N\T(\La)) = \Eq_\be^{\fty}(\N\T(\La)) \text{ for } \be \notin \{\log \la \mid 1 \leq \la \in \fS\}.
\end{equation}
In particular the strong entropy $\be_c$ is indeed a phase transition (the maximal one). 
We also describe the gauge-invariant KMS-simplex at $\log \la$ for $\la \in \fS$ by a weak*-homeomorphism with specific tracial states of the C*-algebra $A \subseteq \N\T(\La)$ generated by the vertices, namely
\begin{equation}
\GEq_{\log \la}^F(\N\T(\La)) \simeq \Tr_{\log \la}^F(A) \textup{ when } 1 \leq \la \in \fS,
\end{equation}
where a characterization of $\Tr_{\log \la}^F(A)$ is given in Proposition \ref{P:F chara}.
For $F = \mt$ this amounts to classifying the KMS-states by the common eigenvectors at $\la$.
On the other hand if $\la < \la'$ are successive in $\fS$ then
\begin{equation}
\Eq_\be^{\fty}(\N\T(\La)) = \GEq_\be^{\fty}(\N\T(\La)) \simeq \sca{\de_v \mid \rho\left(\overleftarrow{\La}(v)\right) \leq \la} \neq \mt
\text{ when }
0 < \be \in (\log \la, \log \la'].
\end{equation}
Here $\overleftarrow{\La}(v)$ denotes the higher-rank subgraph consisting of the vertices with which the vertex $v$ (forwards) communicates, and all edges connecting them.
We conclude that the parametrizations are weak*-homeomorphisms.

It is worth mentioning that the connection we obtain here is between the KMS-states of $\N\T(\La)$ and geometric information on the vertices of the graph, rather than with their restriction on the subalgebra of the vertices $A$.
In Remark \ref{R:Chr} we derive a form of the KMS-states similar to the one obtained by Christensen \cite{Chr18}.

In Remark \ref{R:descends} we show how the weak*-homeomorphisms descend to $\N\O(\La)$.
For the finite part we have that
\begin{equation}
\GEq_\be^{\fty}(\N\O(\La)) \simeq \sca{v \in V \mid v \textup{ is a source}, \rho\left(\overleftarrow{\La}(v)\right) \leq \la}
\text{ when }
\be \in (\log \la, \log \la'].
\end{equation}
For the non-finite parts at $\log \la > 0$ with $1 \leq \la \in \fS$ we have that if $\mt \neq F \neq \{1, \dots, N\}$ then
\begin{equation}
\GEq_{\log \la}^F(\N\O(\La)) \simeq \big\{\tau \in \Tr_{\log \la}^F(A) \mid \supp \tau \subseteq \{v \mid \text{ $v$ is not $F$-tracing}\} \big\}
\end{equation}
while for $F = \mt$ the infinite-type part descends as is to $\N\O(\La)$.
Here a vertex $v$ is not $F$-tracing if either it is an $F$-source itself or it is $F^c$-communicated by an $F$-source.

As seen from the examples it may be the case that $\N\O(\La)$ does not admit any KMS-state for the original action.
However commuting nonnegative matrices have a common nonnegative eigenvector at possibly different eigenvalues $\la_i$.
One may thus normalize the action by weights $s_i := \la_i$ (as long as they are non-zero).
Then $\N\O(\La)$ admits at least one KMS-state of infinite type at the normalized $\be = 1$.
Of course one is free to normalize along different weights $(s_1, \dots, s_N)$.
The theory herein applies to these cases simply by recalibrating the entropies by a factor of $s_i$.

In Section 2 we introduce notation and translate the main points of \cite{HLRS14, Chr18, Kak18} in the higher-rank context.
In Section 3 we investigate the phase transitions and provide the main conclusions.
In Section 4 we apply the theory to several examples from the literature.
In Section 5 we close with a note on weighted dynamics showing how our theory applies there as well.

\section{Preliminaries}\label{S:pre}

Let us begin by fixing notation for higher-rank graphs and their associated product systems. 
For more details the reader is addressed to \cite{RS03, RSY03, RSY04} (here we just consider finite graphs).

We will denote the generators of $\bZ_+^N$ by $\bo{1}, \dots, \Bi, \dots, \bo{N}$, where $\bZ_+$ denotes the set of nonnegative integers.
For $F \subseteq \{1, \dots, N\}$ we write
\[
\un{n} =(n_1, \dots, n_N) \in F \qiff \{i \mid n_i \neq 0\} \subseteq F.
\]
We say that $\un{n} \perp \un{m}$ when they are supported on disjoint directions.

Let $G= (V,E,r,s)$ be a directed graph with $r(\mu)$ (resp., $s(\mu)$) denoting the terminal (resp., initial) vertex of a path $\mu$.
It is important to note that paths are read from right to left to comply with operator multiplication.
Denote by $E^{\bullet}$ the collection of all paths in $G$ and partition the edge set
\[
E = E_1 \bigcupdot \cdots \bigcupdot E_N,
\]
such that each edge carries a unique colour from a selection of $N$ colours. 
If $n_i$ is the number of edges in $\mu$ from $E_i$, we define the \emph{multi-degree} $\ell(\mu)$ and the \emph{length} $|\mu|$ by
\[
\ell(\mu) := (n_1, \dots, n_N)
\qand
|\mu| := n_1 + \cdots + n_N.
\]
A \emph{higher rank $N$-structure} on $G$ is an equivalence relation $\sim$ on $E^{\bullet}$ such that for all $\lambda \in E^{\bullet}$ and $\un{m},\un{n} \in \bZ_+^N$ with $\ell(\lambda) = \un{m} + \un{n}$, there exist unique $\mu,\nu \in E^{\bullet}$ with $s(\lambda) = s(\nu)$ and $r(\lambda) = r(\mu)$, such that $\ell(\mu)=\un{n}$ and $\ell(\nu) = \un{m}$ and $\lambda \sim \mu \nu$. 
That is now, there is one way of going from one vertex to the other up to shuffling colour-wise.
This is also referred to as the \emph{factorization property}.
We write $\La := E^{\bullet} / \sim$ and keep denoting by $\ell$ and $|\cdot|$ the induced multi-degree maps on $\La$. 
It is usual to still denote by $\mu, \la$ etc.\ the elements of $\La$, but in order to make a distinction we will write $\umu, \un{\la}$ etc.\ for the representatives of the equivalence classes.
In this way the pair $\La$ is a \emph{higher-rank graph} as in \cite[Definition 2.1]{RSY03}. 

A higher-rank graph is \emph{finite} if the set of vertices and the set of edges are finite.
In this case $\La$ constitutes of $N$ commuting matrices with prescribed (associative) commutation relations on squares on different colours, and we write this by
\[
\La = (\La^{(1)}, \dots, \La^{(N)}; \sim).
\]
Here we adopt the convention that $\La^{(i)}_{ts}$ denotes the number of the $\{i\}$-coloured edges from $v_s$ to $v_t$.

\begin{remark}
A higher-rank graph on two colours, say red and blue, amounts to having the same number of blue-red paths and red-blue paths for any pair of vertices.
Of course their matching each time produces a ``different'' higher-rank graph. 
However when we have three colours or more then the commutation relations must also be associative, i.e., independent of the order of the colours appearing in a path.
\end{remark}

For $\un{n} \in \bZ_+^N$ we write $\La^{\un{n}} := \{\umu \in \La \mid \ell(\umu) = \un{n}\}$.
Notice here the difference between the set $\La^{\Bi}$ of the $i$-coloured edges and the matrix $\La^{(i)}$ of the $i$-coloured subgraph.
For arbitrary $\un{n}$ we will write
\begin{equation}
\La^{(\un{n})} := \prod_{i \in \supp \un{n}} [\La^{(i)}]^{n_i}.
\end{equation}
For $\umu \in \La$ and $\S \subseteq \La$, we define
\begin{equation}
\umu \S := \{\umu \, \un{\nu} \in \La \mid \un{\nu} \in \S \} \qand \S \umu := \{\un{\nu} \, \umu \in \La \mid \un{\nu} \in \S \}.
\end{equation}
We say that a vertex $v$ is an \emph{$F$-source} for some $F \subseteq \{1, \dots, N\}$ if it does not receive any path $\umu \in \La$ with $\ell(\umu) \in F$.
Equivalently, if $v \La^{\Bi} = \mt$ for all $i \in F$, due to the unique factorization property.
We say that a vertex $v$ is an \emph{eventual $F$-source} if there exists a $k_0 \in \bN$ such that 
\[
v \La^{\un{n}} = \mt \text{ whenever $\un{n} \in F$ with $k_0 \sum_{i \in F} \Bi \leq \un{n}$.}
\]
Equivalently, $D_{v v'} = 0$ for all $v' \in V$ for the matrix
\[
D : = [\prod_{i \in F} \La^{(i)}]^{k_0} = \prod_{i \in F} (\La^{(i)})^{k_0}.
\]
When $F = \{1, \dots, N\}$ we simply say that $v$ is a \emph{source}, resp. an \emph{eventual source}.

In order to construct the Toeplitz-Cuntz-Krieger C*-algebra we need to take into account common backwards extensions of paths.
For $\un{\la}, \umu \in \La$ let
\begin{equation}
\La^{\min}(\un{\la}, \umu) := \big\{ (\un{\al}, \un{\al}') \mid \un{\la} \, \un{\al} = \umu \, \un{\al}', \ell(\un{\la} \, \un{\al}) = \ell(\un{\la}) \vee \ell(\umu) = \ell(\umu \, \un{\al}') \big\}
\end{equation}
be the set of \emph{minimal common extensions of $\un{\la}$ and $\umu$}.
A higher-rank graph is called \emph{finitely aligned} if $|\La^{\min}(\un{\la}, \umu)| < \infty$ for all $\un{\la}, \umu \in \La$.
Finite graphs are finitely aligned.
A set of partial isometries $\{T_{\umu}\}_{\umu \in \La}$ is called a \emph{Toeplitz-Cuntz-Krieger $\La$-family} for a finitely aligned $\La$, if it satisfies the following three conditions: \vspace{3pt}
\begin{enumerate}
\item[(P)]
$\{T_v\}_{v\in \La^{\un{0}}}$ is a collection of pairwise orthogonal projections; \vspace{3pt}
\item[(HR)]
$T_{\un{\la}} T_{\umu} = \de_{s(\un{\la}), r(\umu)} T_{\un{\la} \, \umu}$ for all $\un{\la}, \umu \in \La$; and \vspace{3pt}
\item[(NC)]
$T_{\un{\la}}^* T_{\umu} = \sum \big\{ T_{\un{\al}} T_{\un{\al}'}^* \mid (\un{\al}, \un{\al}')\in \La^{\min}(\un{\la}, \umu) \big\}$ for all $\un{\la}, \umu \in \La$. \vspace{3pt}
\end{enumerate}
Note that (NC) suggests that $T_{\un{\la}}^* T_{\un{\mu}} = \de_{\un{\la}, \un{\mu}} T_{s(\un{\la})}$ whenever $\ell(\un{\la}) = \ell(\un{\mu})$.
We write $\N\T(\La)$ for the universal C*-algebra of the Toeplitz-Cuntz-Krieger $\La$-families for a finitely aligned graph.

In \cite{DK18} the author with Dor-On provided an alternative way to visualize the Cuntz-Krieger quotient of $\N\T(\La)$ when $\La$ is \emph{strong finitely aligned}, i.e., when in addition
\[
|\{\La^{\min}(\un{\la}, \un{e}) \mid \ell(\un{e}) = \Bi \perp \ell(\un{\la})\}| <\infty \foral \un{\la} \in \La \text{ and } \Bi \perp \ell(\un{\la}).
\]
Finite graphs are automatically strong finitely aligned.
Let $\mt \neq F \subseteq \{1, \dots, N\}$ be a set of directions.
A vertex $v \in \La^{\un{0}}$ is called \emph{$F$-tracing} if: \vspace{2pt}
\begin{enumerate}
\item for every $\umu \in r^{-1}(v) \cap \ell^{-1}(F^c)$ there is an $i \in F$ such that $s(\umu)\La^{\bo{i}} \neq \emptyset$; and \vspace{2pt}
\item $|s(\umu)\La^{\bo{i}}| < \infty$ for all $i \in \{1, \dots, N\}$. \vspace{2pt}
\end{enumerate}
The second condition is redundant when $\La$ is finite.
In \cite{DK18} it is shown that a Toeplitz-Cuntz-Krieger $\La$-family of a strong finitely aligned higher-rank graph is a \emph{Cuntz-Krieger $\La$-family} in the sense of Raeburn-Sims-Yeend \cite{RSY03} if for every $\mt \neq F \subseteq \{1, \dots, N\}$ it satisfies: \vspace{3pt}
\begin{enumerate}
\item[(CK)]
$T_v \prod \{ I - T_{\un{e}}T_{\un{e}}^* \mid \un{e} \in \La^{\Bi}, i \in F\} = 0$ for every $F$-tracing vertex $v$. \vspace{3pt}
\end{enumerate}
We write $\N\O(\La)$ for the universal C*-algebra of the Cuntz-Krieger $\La$-families for a strong finitely aligned graph.

The Toeplitz-Cuntz-Krieger algebra admits a gauge action of the $N$-torus.
This in turn defines a rotational action and consequently a theory of equilibrium states induced by \cite{BR97}.
An Huef-Laca-Raeburn-Sims \cite{HLRS14} initiated their study and let us provide here the fundamental definitions.
In short, let $\{\ga_{\un{z}}\}_{\un{z} \in \bT^N}$ be the gauge action on $\N\T(\La)$ such that
\[
\ga_{\un{z}}(T_{\umu}) = \un{z}^{\ell(\umu)} T_{\umu} \foral \umu \in \La,
\]
and let
\[
\si \colon \bR \to \Aut(\N\T(\La)): r \mapsto \ga_{(\exp(i r), \dots, \exp(i r))}.
\]
The monomials of the form $T_{\un{\la}} T_{\umu}^*$ span a dense $\si$-invariant $*$-subalgebra of analytic elements of $\N\T(\La)$ since the function
\[
\bR \to \N\T(\La): r \mapsto \si_r(T_{\un{\la}} T_{\umu}^*) = e^{i (|\un{\la}| - |\umu|) r} T_{\un{\la}} T_{\umu}^*
\]
is analytically extended to the entire function
\[
\bC \to \N\T(\La) : z \mapsto e^{i (|\un{\la}| - |\umu|) z} T_{\un{\la}} T_{\umu}^*.
\]
We say that a state $\vphi$ satisfies the \emph{$(\si,\be)$-KMS condition} at $\be > 0$ if
\begin{equation*}\label{eq:kms}
\vphi(T_{\un{\la}} T_{\umu}^* \cdot T_{\un{\la}'} T_{\umu'}^*)
=
e^{-(|\un{\la}| - |\umu|) \be} \vphi(T_{\un{\la}'} T_{\umu'}^* \cdot T_{\un{\la}} T_{\umu}^*)
\foral
\un{\la}, \umu, \un{\la}', \umu' \in \La.
\end{equation*}
We write $\Eq_\be(\N\T(\La))$ for the simplex of the KMS-states at $\be >0$.
If $E \colon \N \T(\La) \to \N\T(\La)^\ga$ is the conditional expectation induced by $\{\ga_{\un{z}}\}_{\un{z} \in \bT^N}$, then we write
\[
\GEq_\be(\N\T(\La)) := \{ \vphi \in \Eq_\be(\N\T(\La)) \mid \vphi = \vphi E\}
\]
for the sub-simplex of the gauge-invariant equilibrium states.

The equilibrium states of higher-rank graphs have been under thorough examination in the past years by an Huef-Laca-Raeburn-Sims \cite{HLRS14, HLRS15}, an Huef-Kang-Raeburn \cite{HKR17} and Fletcher-an Huef-Raeburn \cite{FHR18} under some conditions on the graph.
Lately Christensen \cite{Chr18} removed all conditions and provided a description of the simplex at inverse temperature $\be >0$.
Finite higher-rank graphs form product systems with finite frames that were considered by the author \cite{Kak18}.
Let us summarize here the main points of \cite{Kak18} that will help with the analysis of the phase transitions.
To this end, let 
\begin{equation*}
\N\T(\La) = \ca(T_{\umu} \mid \umu \in \La)
\end{equation*}
for a finite higher-rank graph $\La$, and define the projections
\begin{equation*}
1 - Q_{\Bi} = P_\Bi := \sum_{\ell(e) = \Bi} T_{e} T_{e}^*, \quad
Q_F := \prod_{i \in F} (1 - P_\Bi) \qand
Q_F^{\un{n}} := \sum_{\ell(\umu) = \un{n}} T_{\umu} Q_F T_{\umu}^*,
\FOR
\un{n} \in F.
\end{equation*}
For every $F \subseteq \{1, \dots, N\}$ we define the \emph{$F$-subharmonic simplex}
\begin{equation}
\Eq_\be^F(\N\T(\La)) := \{\vphi \in \Eq_\be(\N\T(\La)) \mid \sum_{\un{n} \in F} \vphi(Q_{F}^{\un{n}}) = 1 \text{ and } \vphi(Q_\Bi) = 0 \foral i \notin F\},
\end{equation}
with the understanding that, for $F = \{1, \dots, N\}$ we have the \emph{finite-type simplex}
\begin{equation}
\Eq_\be^{\fty}(\N\T(\La)) 
\equiv
\Eq_\be^{\{1, \dots, N\}}(\N\T(\La))
:= \{\vphi \in \Eq_\be(\N\T(\La)) \mid \sum_{\un{n} \in \bZ_+^N} \vphi(Q^{\un{n}}_{\{1, \dots, N\}}) = 1 \},
\end{equation}
and for $F = \mt$ we have the \emph{infinite-type simplex}
\begin{equation}
\Eq_\be^\infty(\N\T(\La)) 
\equiv
\Eq_\be^{\mt}(\N\T(\La))
:= 
\{\vphi \in \Eq_\be(\N\T(\La)) \mid \vphi(Q_\Bi) = 0 \foral i =1, \dots, N \}.
\end{equation}
By construction every finite $(\si,\be)$-KMS state is automatically gauge-invariant.
Then \cite[Theorem A]{Kak18} asserts that for every $\vphi \in \Eq_\be(\N\T(\La))$ there are $\vphi_F \in \Eq_\be^F(\N\T(\La))$ and $\la_F \in [0,1]$ with $\sum_F \la_F = 1$ such that $\vphi = \sum_F \la_F \vphi_F$.
It must be noted that, in the higher-rank graph context, this decomposition has been first established by Christensen \cite{Chr18}.

In \cite[Theorem B]{Kak18} we established a correspondence between each subharmonic part and a simplex of states on the diagonal
\begin{equation}
A := \ca(T_v \mid v \in V) \simeq \bC^{|V|}.
\end{equation}
For every $\mt \neq F \subseteq \{1, \dots, N\}$ and $\tau \in \Tr(A)$ a (tracial) state on $A$ let
\begin{equation}\label{eq:c}
c_{\tau, \be}^F 
:= 
\sum \big\{ e^{- |\umu| \be} \tau(T_{\umu}^* T_{\umu}) \mid \ell(\umu) \in F \big\}
=
\sum \big\{ e^{- |\umu| \be} \tau(T_{s(\umu)}) \mid \ell(\umu) \in F \big\}.
\end{equation}
Then the associated $F$-set is given by
\begin{equation}\label{eq:tr}
\Tr_{\be}^F(A)
:=
\{ \tau \in \Tr(A) \mid c_{\tau,\be}^F < \infty
\text{ and }
\La^{(i)} \tau = e^{\be}\tau \foral i \notin F\},
\end{equation}
where we make the association of a trace on $A$ with an $\ell^1$-unimodular nonnegative vector on $|V|$-dimensions.
In particular for $F = \{1, \dots, N\}$ we write
\begin{equation}
\Tr_\be^{\fty}(A) := \big\{ \tau \in \Tr(A) \mid c_{\tau, \be}^{\{1, \dots, N\}} = \sum \{ e^{- |\umu| \be} \tau(T_{s(\umu)}) \mid \umu \in \La \} < \infty \big\}.
\end{equation}
For $F = \mt$ we define 
\begin{equation}
\Avt_\be(A) := \{\tau \in \Tr(A) \mid \La^{(i)} \tau = e^{\be}\tau \foral i = 1, \dots, N\}.
\end{equation}
For the general theory of \cite{Kak18}, we need to further restrict to traces that annihilate the ideal
\[
\fJ_{F^c} := \ker\big\{ A \hookrightarrow \N\T(\La) \rightarrow \N\T(\La)/\sca{Q_{\Bi} \mid i \notin F} \big\}.
\]
By construction we have that $T_v \in \fJ_{F^c}$ if and only if there exists a $k_0 \in \bN$ such that 
\[
T_v T_{\umu} = 0 \textup{ whenever $\ell(\umu) \in F^c$ with $k_0 \sum_{i \notin F} \Bi \leq \ell(\umu)$}.
\]
Equivalently, if and only if $v$ is eventually an $F^c$-source, and so
\[
\fJ_{F^c} = \ca(T_v \mid \textup{$v$ is eventually an $F^c$-source}).
\]
Then \cite[Theorem B]{Kak18} asserts that there is a weak*-homeomorphism
\[
\Psi \colon \GEq_\be^{\infty}(\N\T(\La)) \to \{\tau \in \Avt_\be(A) \mid \tau(T_v) = 0\text{ when $v$ is eventually a source}\},
\]
and that there is a weak*-continuous bijection 
\[
\Phi^F \colon \GEq_\be^{F}(\N\T(\La)) \to \{ \tau \in \Tr_\be^F(A) \mid  \tau(T_v) = 0 \textup{ when $v$ is eventually an $F^c$-source}\}.
\]
These parametrizations respect convex combinations and thus the extreme points of the simplices.
We will show in Theorem \ref{T:main gi} that actually every $\tau \in \Tr_\be^F(A)$ automatically annihilates every eventual $F^c$-source and that every $\Phi^F$ is a weak*-homeomorphism.

In Theorem \ref{T:main gi} we will use the form of the parametrizations $\Psi$ and $\Phi^F$ and a note is in place about how they are constructed.
If $\vphi \in \GEq_\be^{\infty}(\N\T(\La))$ then we define
\[
\Psi(\vphi)(T_v) = \vphi(T_v) \foral v \in V,
\]
and if $\vphi \in \GEq_\be^F(\N\T(\La))$ for $F \neq \mt$ then we define
\[
\Phi^F(\vphi)(T_v) = \vphi(Q_F)^{-1} \vphi(Q_F T_v Q_F) \foral v \in V.
\]
Conversely, if $\tau \in \Tr_\be^F(A)$ annihilates the eventual $F^c$-sources then it can be extended (homeomorphically) to a KMS-state $\wt\tau$ on the subalgebra generated by $\{\La^{(i)}\}_{i \notin F}$, in the sense that
\begin{align*}
\wt{\tau}(T_{\un{\la}} T_{\un{\la}'}^*)
& =
\begin{cases}
e^{-|\un{\la}| \be} \sum_{\ell(\umu) = \un{n}} \tau(\sca{T_{\umu}^* T_{\un{\la}} T_{\un{\la}'}^* T_{\umu}})
& \text{if }
\ell(\un{\la}) = \ell(\un{\la}') = \un{n} \perp F,\\
0 & \text{if } \ell(\un{\la}), \ell(\un{\la}') \perp F, \ell(\un{\la}) \neq \ell(\un{\la}'),
\end{cases} \\
& =
\de_{\un{\la}, \un{\la}'} e^{-|\un{\la}| \be} \tau(T_{s(\un{\la})})
\end{align*}
when $\ell(\un{\la}), \ell(\un{\la}') \perp F$.
When $F = \mt$ then this gives $\Psi^{-1}$.
If $F \neq \mt$ we may use $F$-supported statistical approximations on $\wt{\tau}$ and finally derive the KMS-state $\vphi_\tau \equiv (\Phi^F)^{-1}(\tau)$ in $\GEq_\be^F(\N\T(\La))$ such that for $\un{\la}, \un{\la}' \in \La$ we obtain
\begin{align*}
\vphi_\tau(T_{\un{\la}} T_{\un{\la}'}^*) 
& = 
\de_{\ell(\un{\la}), \ell(\un{\la}')} \cdot (c_{\tau, \be}^F)^{-1} \cdot \sum_{\ell(\umu) \in F} e^{-(|\un{\la}| + |\umu|) \be} \tau(T_{\umu}^* T_{\un{\la}'}^* T_{\un{\la}} T_{\umu}) \\
& =
\de_{\un{\la}, \un{\la}'} \cdot (c_{\tau, \be}^F)^{-1} \cdot e^{-|\un{\la}| \be} \cdot \sum_{\ell(\umu) \in F} e^{- |\umu| \be} \tau(T_{\umu}^* T_{s(\un{\la})} T_{\umu}),
\end{align*}
where we used that $\de_{\ell(\un{\la}), \ell(\un{\la}')} T_{\un{\la}'}^* T_{\un{\la}} = \de_{\un{\la}, \un{\la}'} T_{\un{\la}'}^* T_{\un{\la}} = \de_{\un{\la}, \un{\la}'} T_{s(\un{\la})}$.
In Remark \ref{R:Chr} we will derive the connection with the form of $\vphi_\tau$ as established by Christensen \cite{Chr18}.

This parametrization descends to $\N\O(\La)$ when the trace annihilates the $F$-tracing vertices in the sense that
\[
\GEq_{\be}^{\infty}(\N\O(\La)) \simeq \GEq_{\be}^{\infty}(\N\T(\La)),
\]
while 
\[
\GEq_{\be}^F(\N\O(\La)) \simeq \big\{\tau \in \Tr_{\be}^F(A) \mid \tau(T_v) = 0 \text{ when $v$ is $F$-tracing or an eventual $F^c$-source} \big\}
\]
for $\mt \neq F \subsetneq \{1, \dots, N\}$, and
\[
\GEq_{\be}^{\fty}(\N\O(\La)) \simeq \big\{\tau \in \Tr_{\be}^{\fty}(A) \mid \tau(T_v) = 0 \text{ when $v$ is not a source} \big\}.
\]
However these simplices may be empty now.

It is worth mentioning that these parametrizations are essentially different from what is obtained in \cite{ALN18, Chr18, FHR18, HKR17, HLRS14, HLRS15, LN04}.
The parametrizations therein link to what a KMS-state does by restriction on $A$.
In \cite{Kak17, Kak18} we construct KMS-states from geometric information on the vertices and the matrices, related to some specific series convergence.
By re-formulating the terminology of \cite{Kak18}, we define the \emph{strong $F$-entropy} of $\La$ by
\begin{equation}
h_\La^{s,F} := \limsup_k \frac{1}{k} \log \| \sum_{\ell(\umu) \in F, |\umu| = k} T_{\umu}^* T_{\umu} \|,
\end{equation}
and we write $h_\La^{s}$ when $F = \{1, \dots, N\}$.
The \emph{tracial $F$-entropy} of a $\tau \in \Tr(A)$ is given by
\begin{equation}
h_\La^{\tau, F} := \limsup_k \frac{1}{k} \log \bigg( \sum_{\ell(\umu) \in F, |\umu| = k} \tau(T_{\umu}^* T_{\umu}) \bigg),
\end{equation}
and we write $h_\La^{\tau}$ when $F = \{1, \dots, N\}$.
The root test links $h_\La^{\tau, F}$ with the constant $c_{\tau, \be}^F$ of equation (\ref{eq:c}).
In \cite[Theorem D]{Kak18} it is shown that
\[
h_\La^\tau \leq h_\La^s = \max \{ h_\La^{s, i} \mid i =1, \dots, N \},
\]
and that
\[
h_\La^{\tau, F} \leq h_\La^{\tau} \leq \be
\foral
\tau \in \Tr_\be^F(A).
\]
Moreover it is shown that there are no equilibrium states at $\be>0$ when $\be$ is less than
\[
h_\La = \max \big\{0, \inf \{h_\La^{\tau} \mid \tau \in \Tr(A)\} \big\},
\]
and $h_\La$ is optimal.
It has also been shown in \cite{Kak18} and in \cite{ALN18} that there are no non-finite parts above $h_\La^s$, however it has been left open if $h_\La^s$ is optimal in this respect.

\section{Phase transitions}

Up to a permutation of the vertices, every $\La^{(i)}$ can be written in a lower triangular form
\[
\La^{(i)}
=
\begin{bmatrix}
\La^{(i)}_{1} & 0 & \cdots & 0 \\
\ast & \La^{(i)}_{2} & \cdots & 0 \\
\vdots & \vdots & \ddots & \vdots \\
\ast & \ast & \cdots & \La^{(i)}_{m_i}
\end{bmatrix},
\]
where the $\La^{(i)}_1, \dots, \La^{(i)}_{m_i}$ are irreducible components (including the possibility that they are equal to $[0]$); see for example \cite[Equation (4-4-1)]{LM95}.
Then the Perron-Frobenius eigenvalue of $\La^{(i)}$, i.e., the maximum positive eigenvalue, is given by
\[
\rho(\La^{(i)}) = \max \{\rho(\La^{(i)}_1), \dots, \rho(\La^{(i)}_{m_i}) \}.
\]
By the Perron-Frobenius Theorem we have two cases: if $\La^{(i)}_1 = \cdots = \La^{(i)}_{m_i} = [0]$ then $\rho(\La^{(i)}) = 0$; otherwise $\rho(\La^{(i)}) \geq 1$.
Moreover the graph-entropy of $\La^{(i)}$ equals $\log \rho(\La^{(i)})$.
Of course it may not happen that the same permutation works for all colours, but we can always find a permutation that makes all $\La^{(i)}$ lower triangular, simultaneously.
To this end paint back all edges by one-colour and consider the graph with adjacency matrix
\[
G:= \La^{(1)} + \cdots + \La^{(N)}.
\]
We can have a permutation of vertices to achieve a block lower triangular form for $G$ as
\[
G
=
\begin{bmatrix}
G_{1} & 0 & \cdots & 0 \\
\ast & G_{2} & \cdots & 0 \\
\vdots & \vdots & \ddots & \vdots \\
\ast & \ast & \cdots & G_{m}
\end{bmatrix},
\]
where each $G_1, \dots, G_m$ is an irreducible component with respect to all colours.
As the irreducible components of the $\La^{(i)}$ must be contained in some irreducible component of $G$, this permutation induces a lower triangular form for each $\La^{(i)}$ (by re-painting the edges).
A second permutation within each $G_1, \dots, G_m$ (possibly different for every $i$) induces the block lower triangular form for each $\La^{(i)}$.
For convenience we write
\[
\rho(\La) := \max\{ \rho(\La^{(1)}), \dots, \rho(\La^{(N)}) \}.
\]

\begin{proposition}\label{P:strong entropy}
Let $\La = (\La^{(1)}, \dots, \La^{(N)}; \sim)$ be a higher-rank graph.
If $\B_{k, F}(\La)$ denotes the number of paths in $\La$ of length $k$ that are supported on $F$ then
\begin{align*}
h_{\La}^{s,F}
& =
\limsup_k \frac{1}{k} \log \big( \B_{k, F}(\La) \big) 
=
\max\{ \log \rho(\La^{(i)}) \mid i \in F \}.
\end{align*}
\end{proposition}

\begin{proof}
In \cite[Proposition 7.3]{Kak18} and \cite[Theorem 8.9]{Kak17} we have shown respectively that 
\[
h_\La^{s, F} = \max_{i \in F} h_\La^{s, \{i\}}
\qand
h_\La^{s, \{i\}} = \limsup_k \frac{1}{k} \log \B_{k, \{i\}}(\La) = \log \rho(\La^{(i)}).
\]
Since $|\B_{k, F_1 \cupdot F_2}| \leq \sum_{n=0}^k |\B_{n, F_1}| \cdot |\B_{k - n, F_2}|$, an argument similar to that in the proof of \cite[Proposition 7.3]{Kak18} yields
\[
\limsup_k \frac{1}{k} \log \B_{k, F}(\La)
=
\max \{ \limsup_k \frac{1}{k} \log \B_{k, \{i\}}(\La) \mid i \in F \},
\]
and the proof is complete.
\end{proof}

Our next step is to show that there is indeed a nonnegative phase transition whenever $\rho(\La) \geq 1$.
By combining with \cite[Proposition 7.6]{Kak18} we thus get that $\log \rho(\La)$ is then the largest phase transition for $\N\T(\La)$.
As we are looking for nonnegative eigenvectors we restrict our attention to eigencones rather than eigenspaces.

\begin{proposition}\label{P:max phase}
Let $\La = (\La^{(1)}, \dots, \La^{(N)}; \sim)$ be a higher-rank graph.
If $\rho(\La) \geq 1$ then there exists an $F \subsetneq \{1, \dots, N\}$ such that $\Tr_{\log \rho(\La)}^F(A) \neq \mt$.
\end{proposition}

\begin{proof}
Without loss of generality suppose that $\rho(\La) = \rho(\La^{(1)}) =: \la_1$.
By \cite[Theorem 8.3.1]{HJ13} the eigencone of $\La^{(1)}$ at $\la_1$ is non-trivial.
Then \cite[Theorem 3.5]{KSS12} asserts that there exists a nonnegative (and non-zero) common eigenvector $w$ for $\La^{(1)}, \La^{(2)}, \dots, \La^{(N)}$ at some eigenvalues $\la_1, \la_2, \dots, \la_N$.
Since $w \geq 0$ and all matrices have nonnegative entries it transpires that every $\la_i$ is nonnegative.
Furthermore \cite[Theorem 8.3.2]{HJ13} asserts that $\la_i \leq \rho(\La^{(i)}) \leq \la_1$.
Let
\[
F := \{i \in \{1, \dots, N\} \mid \la_i < \la_1 \},
\]
and note that $F^c \neq \mt$ as $1 \in F^c$.
Let $\tau$ be the trace corresponding to the $\ell^1$-normalization of $w$.
By definition we have that $\La^{(i)} \tau = \la_1 \tau$ for all $i \notin F$, while
\[
h_\La^{\tau, F} \leq h_\La^{s, F} = \max_{i \in F} h_\La^{s, \{i\}} = \max_{i \in F} \log \la_i < \log \la_1.
\]
That is $c_{\tau, \log \la_1}^F < \infty$, and so $\tau \in \Tr_{\log \la_1}^F(A)$.
\end{proof}

The $F$-entropy of a trace depends on the irreducible components with which its support communicates.
We introduce some terminology to this end.

\begin{definition}
Let $\La = (\La^{(1)}, \dots, \La^{(N)}; \sim)$ be a higher-rank graph.
A subgraph $H = (H^{(1)}, \dots, H^{(N)}; \sim)$ is said to be a \emph{sink} if there exists a permutation of vertices so that 
\[
\La^{(i)}
=
\begin{bmatrix}
\ast & 0 \\
\ast & H^{(i)}
\end{bmatrix}
\foral
i=1, \dots, N.
\]
\end{definition}

Note that the permutation of the vertices is one and the same for all $i$.
Every $\La$ has at least itself as a sink subgraph.
It is clear that every eigenvector of $H^{(i)}$ extends to an eigenvector of $\La^{(i)}$ (at the same eigenvalue) by adding zeroes.

\begin{definition}
Let $\La = (\La^{(1)}, \dots, \La^{(N)}; \sim)$ be a higher-rank graph.
For $V' \subseteq V$ write $\overleftarrow{\La}(V')$ for the subgraph with vertex set
\[
\{v \in V \mid v \La v' \neq \mt \text{ for some } v' \in V' \}
\] 
and edge set containing all the possible edges of $\La$ connecting those vertices.
For $F \subseteq \{1, \dots, N\}$ we write $\overleftarrow{\La}(V'; F)$ for the subgraph defined over the $F$-coloured paths.
\end{definition}

The left arrow reflects the way we read paths from right to left.
It follows that $\overleftarrow{\La}(V')$ and $\overleftarrow{\La}(V';F)$ are higher-rank graphs; in particular each one of these is a sink subgraph of $\La$.
In the following proposition we connect the $F$-tracial entropy of a $\tau$ in $\La$ with the $F$-strong entropy of the (sink) forward subgraph defined by $\supp \tau$.

\begin{proposition}\label{P:F chara}
Let $\La = (\La^{(1)}, \dots, \La^{(N)}; \sim)$ be a higher-rank graph and $\be>0$.
For $\tau \in \Tr(A)$ and $F \subseteq \{1, \dots, N\}$ let the sink subgraphs 
\[
H := \overleftarrow{\La}(\supp \tau; F)
\qand
G := \overleftarrow{\La}(\supp \tau).
\]
Then
\[
h_H^{s, F} = h_H^{\tau, F} = h_G^{\tau, F} = h_\La^{\tau, F}.
\]
Moreover the following are equivalent:
\begin{enumerate}
\item $\tau$ is in $\Tr_\be^F(A)$;
\item $h_\La^{\tau, F} < \be$ and $\La^{(i)} \tau = e^{\be} \tau$ for all $i \notin F$;
\item $\rho(H^{(i)}) < e^{\be}$ for all $i \in F$, $G^{(i)} \tau|_G = e^{\be} \tau|_G$ for all $i \notin F$, and $\tau(T_v) = 0$ for all $v \notin G$.
\end{enumerate}
\end{proposition}

\begin{proof}
Set  $\la := \max\{\rho(H^{(i)}) \mid i \in F\}$.
Since $\tau(T_{\umu}^* T_{\umu}) = \tau(T_{s(\umu)})$ is zero unless $s(\umu) \in \supp \tau$, we get
\begin{align*}
\sum_{\ell(\umu) \in F, |\umu| = k} \tau(T_{\umu}^* T_{\umu})
& =
\sum_{\ell(\umu) \in F, |\umu| = k, s(\umu) \in \supp \tau} \tau(T_{\umu}^* T_{\umu}) \\
& =
\sum_{\umu \in H, |\umu| = k, s(\umu) \in \supp \tau} \tau(T_{\umu}^* T_{\umu})
=
\sum_{\umu \in H, |\umu| = k} \tau(T_{\umu}^* T_{\umu}),
\end{align*}
where we used that $\ell(\umu) \in F$ and $s(\umu) \in \supp \tau$ if and only if $\umu \in H$ and $s(\umu) \in \supp \tau$, as $H$ is forwards $F$-defined by $\supp \tau$.
Therefore, by Proposition \ref{P:strong entropy}, on one hand we have
\begin{align*}
h_{\La}^{\tau, F}
=
h_{H}^{\tau, F}
\leq
h_{H}^{s, F}
 =
\max \{ h_{H}^{s, \{i\}} \mid i \in F \} 
 =
\max\{\log \rho(H^{(i)}) \mid i \in F\}
=
\log \la.
\end{align*}
If $\la = 0$ then every $H^{(i)}$ is nilpotent and thus $\B_{k,F}(H) = \mt$ eventually.
In this case $h_H^{s,F} = h_H^{\tau, F} = h_\La^{\tau, F} = - \infty$.
Now suppose that $\la \neq 0$ and let $H^{(i_0)}_p$, with $i_0 \in F$, be the irreducible component of some $H^{(i_0)}$ so that
\[
\la = \rho(H^{(i_0)}_p) = \rho(H^{(i_0)}).
\]
Let $v_0$ be a vertex in the support of $\tau$ that connects with a vertex $v_s$ in $H^{(i_0)}_p$ through an $F$-coloured path of length $N_0$.
By the Perron-Frobenius theory on $H^{(i_0)}_p$ there exists an $M > 0$ such that
\[
\sum_{t \in H} [ (H^{(i_0)}_p)^{k-N_0} ]_{t s} \geq M \cdot \la^{k- N_0}.
\]
See for example the comments preceding \cite[Proposition 4.2.1]{LM95}.
The $F$-coloured paths of length $k$ starting at $v_0$ are more than the $\{i_0\}$-coloured paths of length $k-N_0$ that start at $v_s$ and terminate inside $H^{(i_0)}_p$.
Thus we get
\begin{align*}
\sum_{\ell(\umu) \in F, |\umu| = k} \tau(T_{\umu}^* T_{\umu})
& \geq
\tau(T_{v_0}) \sum_{t \in H} [ (H^{(i_0)}_p)^{k-N_0} ]_{t s}
\geq
(\tau(T_{v_0}) M \la^{-N_0}) \cdot \la^{k}.
\end{align*}
Hence $h_{\La}^{\tau, F} \geq \log \la$, which shows that
\[
h_H^{s, F} = h_H^{\tau, F} = h_{\La}^{\tau, F}.
\]
We can apply the second equality for $G$ in place of $\La$ and get that
\[
h_{G}^{\tau, F} = h_{H}^{\tau, F} =h_{\La}^{\tau, F},
\]
which completes the proof of the first part.
Now we move on to prove the equivalences of items (i), (ii) and (iii).

\smallskip

\noindent [(i) $\Leftrightarrow$ (ii)]: 
Every $\tau \in \Tr_\be^F(A)$ must come from a common eigenvector of the $\{\La^{(i)}\}_{i \notin F}$ at $e^{\be}$.
So we just need to check what happens with $c_{\tau, \be}^F$.
By the root test we have that $c_{\tau, \be}^F < \infty$ when $h_\La^{\tau, F} < \be$.
For the converse, if $\max_{i \in F} \rho(\La^{(i)})) = 0$ then $h_{\La}^{\tau, F} = -\infty < 0 < \be$, trivially.
Otherwise $h_{\La}^{\tau, F} \in [0, \infty)$ and set 
\[
\la := \exp({h_\La^{\tau, F}}).
\]
From the first part we deduce a $v_0 \in \supp \tau$, an $N_0 \in \bN$ and an $M>0$ such that
\[
\sum_{\ell(\umu) \in F, |\mu| = k} e^{-k \be} \cdot \tau(T_{\umu}^* T_{\umu})
\geq
(\tau(T_{v_0}) M \la^{-N_0}) \cdot (e^{-\be} \la)^{k}.
\]
Hence $\la < e^{\be}$ when $c_{\tau, \be}^F < \infty$, giving the required $h_{\La}^{\tau, F} < \be$.

\smallskip

\noindent [(i) $\Leftrightarrow$ (iii)]: 
Let $\tau \in \Tr_{\be}^F(A)$.
If $i \in F$ then
\[
\log \rho(H^{(i)}) \leq h_H^{s,F} = h_{\La}^{\tau, F} < \be.
\]
The support of $\tau$ sits inside the sink subgraph $G$, by definition of $G$ and so $\tau(T_v) = 0$ for all $v \notin G$.
Hence for the restriction $\tau|_G$ of $\tau$ on $G$, and $i \notin F$  we obtain 
\[
\begin{bmatrix}
0 \\ G^{(i)} \tau|_G 
\end{bmatrix}
=
\begin{bmatrix}
\ast & 0 \\
\ast & G^{(i)} 
\end{bmatrix}
\cdot
\begin{bmatrix}
0 \\ \tau|_G
\end{bmatrix}
=
\La^{(i)} \tau
=
e^\be \tau
=
\begin{bmatrix}
0 \\ e^{\be} \tau|_G 
\end{bmatrix}.
\]
Conversely suppose that $\tau$ satisfies item (iii).
As $G$ is a sink subgraph, the trace $\tau|_G$ produces a common eigenvector for the $\{ \La^{(i)} \}_{i \notin F}$ at $e^{\be}$ by extending by zeroes.
Moreover by the first part and Proposition \ref{P:strong entropy} applied for $H$ we get that
\[
h_{\La}^{\tau, F} = h_{H}^{\tau, F} \leq h_{H}^{s, F} = \max_{i \in F} \log \rho(H^{(i)}) < \be,
\]
and so $c_{\tau, \be}^F < \infty$.
Thus $\tau$ will be in $\Tr_{\be}^F(A)$.
\end{proof}

\begin{definition}
We say that $\N\T(\La)$ has a \emph{nonnegative subharmonic phase transition at $\log \la \geq 0$} if there exists an $F \subsetneq \{1, \dots, N\}$ such that $\Eq_{\log \la}^F(\N\T(\La)) \neq \mt$.
\end{definition}

We have now arrived at the main result of the section.

\begin{theorem}\label{T:main gi}
Let $\La = (\La^{(1)}, \dots, \La^{(N)}; \sim)$ be a higher-rank graph and set
\[
\fS := \{\rho(H) \mid \textup{ $H$ is a sink subgraph of $\La$} \}.
\]
Then $\N\T(\La)$ has nonnegative subharmonic phase transitions exactly at $\{\log \la \mid 1 \leq \la \in \fS\}$.
Furthermore, for all F $\subsetneq \{1, \dots, N\}$ we have that
\[
\GEq_{\log \la}^F(\N\T(\La)) \simeq \Tr_{\log \la}^F(A)
\text{ when }
1 \leq \la \in \fS.
\]
On the other hand if $\la < \la'$ are successive in $\fS$ then
\[
\mt \neq \GEq_\be^{\fty}(\N\T(\La)) \simeq \sca{v \in V \mid \rho\left(\overleftarrow{\La}(v)\right) \leq \la}
\foral 0 < \be \in (\log \la, \log \la'].
\]
\end{theorem}

\begin{proof}
As phase transitions occur at common eigenvalues of subsets of $\{\La^{(i)}\}_{i=1}^N$ they are isolated points.
On the other hand the finite-part is always gauge-invariant by \cite[Proposition 3.4]{Kak18} and we will show that it stays stable in the induced half-open half-closed intervals.
The conditional expectation dominates the rotational action and therefore we have that $\Eq_\be^F(\N\T(\La)) \neq \mt$ if and only if $\GEq_\be^F(\N\T(\La)) \neq \mt$.
Hence the phase transitions of the $\Eq$-simplex coincide with those of the $\GEq$-simplex, to which we now restrict.

First we show that every $\tau \in \Tr_\be^F(A)$ automatically annihilates the eventual $F^c$-sources.
Hence the parametrization of the simplex $\GEq_\be^F(\N\T(\La))$ is not just in, but it is onto $\Tr_\be^F(A)$.
For convenience suppose that $F = \{N'+1, \dots, N\}$, and let $v \in V$ be an eventual $F^c$-source.
Hence let $k_0 \in \bN$ such that
\[
T_v T_{\umu} = 0 \; \textup{ when } \; k_0 \sum_{i = 1}^{N'} \Bi \leq \ell(\umu) \in \{1, \dots, N'\}.
\]
Equivalently we have that $D_{v v'} = 0$ for all $v' \in V$, for the matrix
\[
D := \prod_{i = 1}^{N'} (\La^{(i)})^{k_0}.
\]
A trace $\tau \in \Tr_\be^F(A)$ is an eigenvector of each $\La^{(i)}$ at $e^\be$ with $i \in \{1, \dots, N'\}$, so that
\begin{align*}
e^{N' k_0 \be} \cdot \tau(T_v) 
& =
\bigg[\prod_{i=1}^{N'} (\La^{(i)})^{k_0} \tau \bigg](T_v) 
=
\big[ D\tau \big](T_v)
=
\sum_{v' \in V} D_{v v'} \tau(T_{v'})
=
0,
\end{align*}
and thus $\tau(T_v) = 0$.

Next we show that $\{\log \la \mid 1 \leq \la \in \fS\}$ is the set of nonnegative subharmonic phase transitions. 
Let $1 \leq \la \in \fS$ and fix $H$ be a sink component of $\La$ for which $\la = \rho(H)$.
We can then use the same order on the vertices to write
\[
\La^{(i)}
=
\begin{bmatrix}
\ast & 0 \\
\ast & H^{(i)}
\end{bmatrix}
\foral
i=1, \dots, N.
\]
By construction the $H^{(i)}$ define a higher-rank graph $H$, and Proposition \ref{P:max phase} induces an $F \subsetneq \{1, \dots, N\}$ and a $\tau$ supported in $H$ with $h_{H}^{\tau, F} < \log \la$ that is a common eigenvector for $\{H^{(i)}\}_{i \notin F}$.
Let $G := \overleftarrow{\La}(\supp \tau)$ and apply Proposition \ref{P:F chara} twice for $G \subseteq \La$ and for $G \subseteq H$ to get
\[
h_{\La}^{\tau, F} = h_{G}^{\tau, F} = h_{H}^{\tau, F} < \log \la,
\]
and that
\[
\La^{(i)} \tau = 
\begin{bmatrix}
0 \\ G^{(i)} \tau|_G 
\end{bmatrix} 
=
\begin{bmatrix}
0 \\ H^{(i)} \tau|_H 
\end{bmatrix}
= \la \tau \foral i \notin F,
\]
where we used that by definition $\supp \tau \subseteq G \subseteq H$.
Hence $\tau \in \Tr_{\log \la}^F(A)$ and thus it induces an $F$-subharmonic KMS-state for $\N\T(\La)$.
Conversely let $\be > 0$ and $F \subsetneq \{1, \dots, N\}$ so that $\GEq_\be^F(\N\T(\La)) \neq \mt$, and consider a $\tau \in \Tr_\be^F(A) \neq \mt$.
As $\tau$ is an eigevector for at least one $\La^{(i_0)}$ we get that $h_{\La}^{\tau, \{i_0\}} = \be$ and so
\[
\be = h_{\La}^{\tau, \{i_0\}} \leq h_\La^\tau \leq \be.
\]
However Propositions \ref{P:strong entropy} and \ref{P:F chara} give
\[
\be = h_\La^\tau = h_{\overleftarrow{\La}(\supp \tau)}^s = \log \rho\left(\overleftarrow{\La}(\supp \tau)\right).
\]
By definition $\overleftarrow{\La}(\supp \tau)$ is a sink subgraph of $\La$, and so $e^\be \in \fS$.

Now let $\la < \la'$ be two successive points in $\fS$ and let a positive $\be \in (\log \la, \log \la']$.
Let $H$ be a sink subgraph with $\rho(H) = \la$ and let $v$ be a vertex in $H$.
Then $\tau = \de_v$ has entropy less or equal than $h_H^s = \log \la < \be$ by Proposition \ref{P:F chara} and so is in $\Tr_\be^{\fty}(A)$.
Thus we have shown that $\Eq_\be^{\fty}(\N\T(\La)) \neq \mt$ and also that
\[
\Tr_\be^{\fty}(A) \supseteq \sca{v \in V \mid \rho\left(\overleftarrow{\La}(v)\right) \leq \la}.
\]
Now let $\tau \in \Tr_\be^{\fty}(A)$.
Then by Proposition \ref{P:F chara} we get that
\[
e^\be > e^{h_\La^{\tau}} = \rho\left(\overleftarrow{\La}(\supp \tau)\right) \in \fS.
\]
Hence $h_\La^{\tau} \neq \la''$ for all $\la''$ with $\la' < \la'' \in \fS$, and so $h_\La^{\tau} \leq \log \la$.
Thus any component that is communicated by $\supp \tau$ has numerical radius at most $\la$ and so 
\[
\supp \tau \subseteq \{v \in V \mid \rho\left(\overleftarrow{\La}(v)\right) \leq \la \}.
\]
Thus $\Tr_\be^{\fty}(A)$ is generated by the vertex set on the right hand side.
As the description depends only on $\la$ we get that $\Tr_\be^{\fty}(A)$ is the same for all $\be \in  (\log \la, \log \la']$.

It is left to show that the parametrization $\Phi^F$ at $\be>0$ is a weak*-homeomorphism.
It suffices to show that it has a weak*-continuous inverse, which can be done as in \cite[Theorem 8.13]{Kak17}.
In short, first fix $F =\{1, \dots, N\}$ and $\be \in (\log \la, \log \la']$, and set 
\[
V' := \{v \in V \mid \rho\left(\overleftarrow{\La}(v)\right) \leq \la \}
\qand
P := \sum_{v \in V'} T_v.
\]
For $\tau \in \Tr(A)$ with $\tau(P) \neq 0$ set $\tau_P(a) := \tau(P)^{-1} \tau(P a P)$ for every $a \in A$; if $\tau(P) = 0$ then set $\tau_P = 0$.
Notice that the definition of $V'$ yields that the tracial entropy of $\tau_P$ is less or equal than $\log \la$.
Thus we get
\[
c_{\tau_P, \be}^{\{1, \dots, N\}} < \infty \foral \tau \in \Tr(A).
\]
Therefore we have
\begin{align*}
\sum \big\{ e^{-|\umu| \be} \tau(T_{\umu}^* a T_{\umu}) \mid s(\umu) \in V', 0 \leq |\umu| \leq k \big\}
& =
\sum \big\{ e^{-|\umu| \be} \tau(P T_{\umu}^* a T_{\umu} P) \mid 0 \leq |\umu| \leq k \big\} \\
& \leq
c_{\tau_P, \be}^{\{1, \dots, N\}} \cdot \nor{a}
< \infty,
\end{align*}
for all $\tau \in \Tr(A)$.
Hence an application of the Banach-Steinhaus Theorem gives that the following limit exists in $A$:
\[
\lim_m \sum_{k=0}^m e^{-k\be} \sum_{|\umu| = k, s(\umu) \in V'} T_{\umu}^* a T_{\umu} = \sum_{k=0}^\infty e^{-k\be} \sum_{|\umu| = k, s(\umu) \in V'} T_{\umu}^* a T_{\umu} \in A.
\]
If $\tau_j \stackrel{\textup{w*}}{\longrightarrow} \tau$ in $\Tr_\be^{\fty}(A)$ then the Monotone Convergence Theorem yields $c_{\tau_j, \be}^{\{1, \dots, N\}} \longrightarrow c_{\tau, \be}^{\{1, \dots, N\}}$ and subsequently that the inverse of $\Phi^{\{1, \dots, N\}}$ is weak*-continous.
For a general $F$ first recall that there is a weak*-continuous map extending $\tau$ to $\wt\tau$ on the $F^c$-part of $\N\T(\La)$.
Applying the previous argument on $\wt\tau$ for the corresponding $F$ yields that the inverse of $\Phi^F$ is weak*-continuous.
\end{proof}

\begin{remark}\label{R:descends}
Let us see how the parametrizations descend to $\N\O(\La)$.
First of all the infinite-type KMS-simplex descends as is to $\N\O(\La)$.
For $\mt \neq F \neq \{1, \dots, N\}$ we have that
\[
\GEq_{\be}^F(\N\O(\La)) \simeq \big\{\tau \in \Tr_{\be}^F(A) \mid \supp \tau \subseteq \{v \mid \text{ $v$ is not $F$-tracing}\} \big\}
\]
for all positive $\be$.
By definition a vertex $v$ is not $F$-tracing if either it is an $F$-source itself or it is $F^c$-communicated by an $F$-source.
In particular we have that a vertex $v$ is not $\{1, \dots, N\}$-tracing if and only if it is a source, and so 
\[
\GEq_\be^{\fty}(\N\O(\La)) \simeq \sca{v \in V \mid v \textup{ is a source}, \rho\left(\overleftarrow{\La}(v)\right) \leq \la}
\]
for all positive $\be \in (\log \la, \log \la']$, when $\la < \la'$ are successive in $\fS$.
\end{remark}

We can use Proposition \ref{P:strong entropy} and Proposition \ref{P:F chara} for computing the $F$-tracial entropy of a state $\tau$ as follows.

\begin{remark}\label{R:tr com}
Suppose that $\tau$ forwards $F$-communicates with the components $\La^{(i)}_{j_1}, \dots, \La^{(i)}_{j_{k_i}}$ of $\La^{(i)}$ for $i \in F$.
Then we have that
\[
h_{\La}^{\tau, F} = h_{H}^{s, F} = \max \{ \log \rho(\La^{(i)}_{j}) \mid j=j_1, \dots, j_{k_i}, i \in F\}
\]
for the sink subgraph $H := \overleftarrow{\La}(\supp \tau; F)$.
\end{remark}

Moreover we can follow the next steps for computing the gauge-invariant KMS-simplices for a higher-rank graph $\La$.
To do so it suffices to compute the possible $\Tr_\be^F(A)$.

\begin{remark}
In what follows we denote by $E(\La^{(i)}, \la_i)$ the nonnegative eigencone of $\La^{(i)}$ at $0 \leq \la_i$, and by $\sca{\S}$ the $\ell^1$-simplex of $\S \subseteq \bR^n_+$.

\medskip

\noindent
{\bf Step I.} Compute the sink subgraphs $H$ of $\La$ and their Perron-Frobenius eigenvalues $\rho(H)$.
By Theorem \ref{T:main gi} these will be the phase transitions for $\N\T(\La)$.

\smallskip

\noindent
{\bf Step II.} Let $\la := \rho(\La)$.
Find the minimal $F \subseteq \{1, \dots, N\}$ such that
\[
\bigcap_{j \in F^c} E(\La^{(j)}, \la) \neq \mt.
\]

\smallskip

\noindent
{\bf Step III.} Fix such an $F$ and set
\[
V^F \equiv \supp \left(\bigcap_{j \in F^c} E(\La^{(j)}, \la)\right) := \{v \in \supp \tau \mid \tau \in \bigcap_{j \in F^c} E(\La^{(j)}, \la) \}.
\]
For every $v \in V^F$ compute its $F$-entropy by using Remark \ref{R:tr com} and set
\[
V_\la^F := \{v \in V^F \mid h_\La^{\de_v, F} < \log \la \}.
\]
Notice here that Proposition \ref{P:max phase} guarantees that there exists at least one $F \subsetneq \{1, \dots, N\}$ and $0 \leq \la_i < \la$ with $i \in F$ so that
\[
\left[ \bigcap_{j \in F^c} E(\La^{(j)}, \la) \right] \cap \left[ \bigcap_{i \in F} E(\La^{(i)}, \la_i) \right] \neq \mt.
\]
For that $F$ and for $v$ in the support of a common nonnegative eigenvector in the above set we have that $h_{\La}^{\de_v, F} \leq \max_{i \in F} \log \la_i <\log \la$.
Therefore there exists an $F$ for which $V_\la^F \neq \mt$ and in particular
\[
\bigcap_{j \in F^c} E(\La^{(j)}, \la) \cap \{\tau \mid \supp \tau \subseteq V_\la^F \} 
\supseteq 
\left[ \bigcap_{j \in F^c} E(\La^{(j)}, \la) \right] \cap \left[ \bigcap_{i \in F} E(\La^{(i)}, \la_i) \right]
\neq \mt.
\]

\smallskip

\noindent
{\bf Step IV.} If $F \neq \mt$ then we get 
\[
\GEq_{\log{\la}}^F(\N\T(\La)) \simeq \Tr_{\log \la}^F(A) = \sca{ \bigcap_{j \in F^c} E(\La^{(j)}, \la) \cap \{\tau \mid \supp \tau \subseteq V_\la^F \}}.
\]
If $F = \mt$ then we get
\[
\GEq_{\log \la}^\infty \simeq \Avt_{\log \la}(A) = \sca{ \bigcap_{j =1}^N E(\La^{(j)}, \la)}.
\]

\smallskip

\noindent
{\bf Step V.} Repeat for all sink subgraphs $H$ of $\La$ to obtain all possible $\Tr_\be^F(A)$ for $\be = \log \rho(H)$.

\smallskip

\noindent
{\bf Step VI.} Use Remark \ref{R:tr com} to compute the tracial entropy of each $\de_v$.
Alternatively compute $\rho\left(\overleftarrow{\La}(v)\right)$.
For $\be \in (\log \la, \log \la']$ with $\la < \la'$ in $\{\rho(H) \mid \text{ $H$ is a sink subgraph of $\La$}\}$ we have
\[
\GEq_{\be}^{\fty}(\N\T(\La)) \simeq \Tr_\be^{\fty}(A) = \sca{\de_v \mid h_\La^{\de_v} \leq \log \la} = \sca{\de_v \mid \rho\left(\overleftarrow{\La}(v)\right) \leq \la}.
\]

\smallskip

\noindent
{\bf Step VII.} 
Each $F$-simplex descends to $\N\O(\La)$ by excluding the traces that are supported on $F$-tracing vertices.
\end{remark}

\begin{remark}\label{R:Chr}
We can further get a form of the gauge-invariant KMS-states similar to the one obtained by Christensen in \cite[Proposition 5.4]{Chr18}.
For large $\be >0$ where only the finite-type part survives this gives a reformulation of the finite-type parametrizations by an Huef-Laca-Raeburn-Sims \cite[Theorem 6.1]{HLRS14}.

Let $\La = (\La^{(1)}, \dots, \La^{(N)}; \sim)$ be a higher-rank graph and $\be>0$.
For $F = \mt$ we have shown that the map
\[
\Psi \colon \GEq_\be^{\infty}(\N\T(\La)) \to \Avt_\be(A)
\]
is a weak*-homeomorphism such that
\[
\Psi^{-1}(\tau)(T_{\un{\la}} T_{\un{\la}'}^*) = \de_{\un{\la}, \un{\la}'} e^{-|\un{\la}|\be}\tau(T_{s(\un{\la})}).
\]

Next we consider the case of $F \neq \mt$.
Suppose first that $F = \{1, \dots, N\}$ and let $0 < \be \in (\log \la, \log \la']$ for $\la, \la' \in \fS$.
Let the higher-rank graph
\[
H := \overleftarrow{\La}(V')
\qfor
V' := \{v \in V \mid \rho(\overleftarrow{\La}(v)) \leq \la\}.
\]
Consequently $\rho(H^{(i)}) \leq \la < e^{\be}$ for all $i = 1, \dots, N$.
Thus we can invoke \cite[Lemma 2.2]{HLRS14} to obtain that the series $\sum_{\un{n} \in \bZ_+^N} e^{-|\un{n}| \be} H^{(\un{n})}$ converges in norm to $\prod_{i = 1}^N (1 - e^{-\be} H^{(i)})^{-1}$.
Let $\tau \in \Tr_\be^{\fty}(A)$ so that $\supp \tau$ is contained in $H$.
It then follows that any path $\umu \in \La$ with $r(\umu) = w \notin H$ must also have $s(\umu) \notin H$.
Consequently we obtain
\begin{align*}
\sum_{v' \in V} \big[ \La^{(\un{n})} \big]_{wv'} \tau(T_{v'})
=
\big[ \La^{(\un{n})} \tau \big] (T_w)
=
\begin{cases}
\big[H^{(\un{n})} \tau|_{H} \big] (T_w) & \text{if } w \in H, \\
0 & \text{if } w \notin H.
\end{cases}
\end{align*}
Therefore for $\vphi_\tau = (\Phi^{\fty})^{-1}(\tau)$ we obtain
\begin{align*}
\vphi_\tau(T_{\un{\la}} T_{\un{\la}'}^*)
& =
\de_{\un{\la}, \un{\la}'} e^{-|\un{\la}| \be} \vphi_{\tau}(T_{s(\un{\la})}) 
 =
\de_{\un{\la}, \un{\la}'} e^{-|\un{\la}| \be} (c_{\tau, \be}^{\{1, \dots, N\}})^{-1}  \sum_{\ell(\umu) \in \bZ_+^N} e^{- |\umu| \be} \tau(T_{\umu}^* T_{s(\un{\la})} T_{\umu}) \\
& =
\de_{\un{\la}, \un{\la}'} e^{-|\un{\la}| \be} (c_{\tau, \be}^{\{1, \dots, N\}})^{-1}  \sum_{\un{n} \in \bZ_+^N} e^{- |\un{n}| \be} \big[ \La^{(\un{n})} \tau \big] (T_{s(\un{\la})}) \\
& =
\begin{cases}
\de_{\un{\la}, \un{\la}'} e^{-|\un{\la}| \be} (c_{\tau, \be}^{\{1, \dots, N\}})^{-1} \big[\prod_{i=1}^N (1 - e^{-\be} H^{(i)})^{-1} \tau|_{H} \big](T_{s(\un{\la})})
& \text{if } s(\un{\la}) \in H, \\
0 & \text{if } s(\un{\la}) \notin H.
\end{cases}
\end{align*}

Now let $\mt \neq F \neq \{1, \dots, N\}$ and a $\tau \in \Tr_\be^F(A)$ for $\be \in \fS$.
Then $\be = \log \la$ for some sink subgraph $G$ of $\La$ with $\rho(G) = \la$ and we can write
\[
\La^{(\un{n})} = \begin{bmatrix} \ast & 0 \\ \ast & G^{(\un{n})} \end{bmatrix}
\foral \un{n} \in \bZ_+^N.
\]
We have seen that $\supp \tau \subseteq G$ and in this case let 
\[
H_\tau:= \overleftarrow{\La}(\supp \tau; F)
\]
which is an $F$-subgraph of $G$ and thus of $\La$.
By Proposition \ref{P:F chara} we have that $\rho(H^{(i)}_\tau) < \la = e^\be$ for $i \in F$, thus the series $\sum_{\un{n} \in F} e^{-|\un{n}| \be} H^{(\un{n})}_\tau$ converges in norm to $\prod_{i \in F} (1 - e^{-\be} H^{(i)}_\tau)^{-1}$.
By the sink subgraph property of $H_\tau$ and that $\supp \tau \subseteq H_\tau$, we now get for $\un{n}$ with $\ell(\un{n}) \in F$ that
\begin{align*}
\sum_{v' \in V} \big[ \La^{(\un{n})} \big]_{wv'} \tau(T_{v'})
=
\begin{cases}
\big[H^{(\un{n})}_\tau \tau|_{H_\tau} \big] (T_w) & \text{if } w \in H, \\
0 & \text{if } w \notin H.
\end{cases}
\end{align*}
Hence a computation as before for $(\Phi^F)^{-1}(\tau) \equiv \vphi_\tau$ and $H_\tau$ gives that
\begin{align*}
\vphi_\tau(T_{\un{\la}} T_{\un{\la}'}^*)
& =
\begin{cases}
\de_{\un{\la}, \un{\la}'} e^{-|\un{\la}| \be} (c_{\tau, \be}^{F})^{-1} \big[\prod_{i \in F} (1 - e^{-\be} H^{(i)}_\tau)^{-1} \tau|_{H_\tau} \big](T_{s(\un{\la})})
& \text{if } s(\un{\la}) \in H_\tau, \\
0 & \text{if } s(\un{\la}) \notin H_\tau.
\end{cases}
\end{align*}
\end{remark}

\section{Examples}

In this section we provide some examples for which we compute the KMS-simplices.
Let us start with the irreducible case that is considered in \cite[Section 7]{HLRS14}.
Unlike to \cite{HLRS14} we will not be weighting the rotational action.
We will revisit this in the next section.

\begin{example}
Let $\La = (\La^{(1)}, \dots, \La^{(N)}; \sim)$ be a higher-rank graph such that every $\La^{(i)}$ is irreducible.
Then there is only one sink subgraph, the entire $\La$.
Let $\la = \rho(\La)$ be achieved at some component, say at $\La^{(1)}$, and let $\tau$ be its Perron-Frobenius eigenvector.
By commutativity and the Perron-Frobenius Theorem we have that $\tau$ is an eigenvector for every $\La^{(i)}$ with
\[
\La^{(i)} \tau = \rho(\La^{(i)}) \tau \foral i=1, \dots, N.
\]
On the other hand all vertices are connected in every $\La^{(i)}$ and thus the entropy of every $\de_v$ equals $\log \la$.
Therefore, by setting
\[
F:= \big\{i \in \{1, \dots, N\} \mid \rho(\La^{(i)}) < \la \big\},
\]
we get
\[
\GEq_{\log \la}(\N\T(\La)) = \GEq_{\log \la}^F(\N\T(\La)) \simeq \{ \tau \}
\]
and that
\[
\Eq_\be(\N\T(\La)) = \Eq_\be^{\fty}(\N\T(\La)) \simeq \sca{\de_v \mid v \in V}
\foral \be > \log \la.
\]
The KMS-simplex of $\N\T(\La)$ is empty for all $\be < \log \la$.

On the other hand we see that every vertex is $F$-tracing for all $F$ as the components are irreducible for every direction.
Thus only the inifnite-type KMS-simplex at $\log \la$ can descend to $\N\O(\La)$.
In particular this is non-empty if and only if $\rho(\La^{(i)}) = \la$ for all $i \in \{1, \dots, N\}$.
\end{example}

Next we study a family of rank-2 graphs considered by Kumjian-Pask \cite{KP00}.
Before we give the complete characterization for those, let us work out two specific examples.

\begin{example}
Let the following 2-coloured graph:
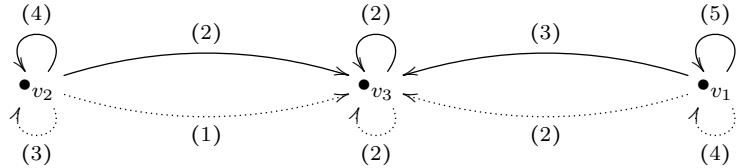
\[
\xymatrix@R=20pt@C=50pt{
\bullet_{v_2} \ar@{->}_{(4)}@(ur,ul) \ar@{->}^{(2)}@/^1pc/[rr] \ar@{..>}^{(3)}@(dr,dl) \ar@{..>}_{(1)}@/_1pc/[rr] & & 
\bullet_{v_3} \ar@{->}_{(2)}@(ur,ul) \ar@{..>}^{(2)}@(dr,dl) & & 
\bullet_{v_1} \ar@{->}_{(5)}@(ur,ul) \ar@{->}_{(3)}@/_1pc/[ll] \ar@{..>}^{(4)}@(dr,dl) \ar@{..>}^{(2)}@/^1pc/[ll]
}
\]
We can write the adjacency matrices with respect to the same order $(v_1, v_2, v_3)$ as
\[
\La^{(1)}
=
\begin{bmatrix}
5 & 0 & 0 \\
0 & 4 & 0 \\
3 & 2 & 2
\end{bmatrix}
\qand
\La^{(2)}
=
\begin{bmatrix}
4 & 0 & 0 \\
0 & 3 & 0 \\
2 & 1 & 2
\end{bmatrix}.
\]
We see that the matrices commute and so they define a rank-2 graph.
We will first analyze the possible sink subgraphs.

\smallskip

\noindent
Step 1. We see that $\rho(\La) = 5$ and $\La^{(1)}$ has an eigenvector
\[
\tau_1 = \frac{1}{2} ( \de_{v_1} + \de_{v_3} )
\]
at $5$.
By checking the support of $\tau_1$ we see that its $\{2\}$-entropy is $\log 4$.
Hence $\tau_1$ contributes one trace in $\Tr_{\log 5}^{\{2\}}(A)$.

\smallskip

\noindent
Step 2. Consider the sink subgraph $H$ given by
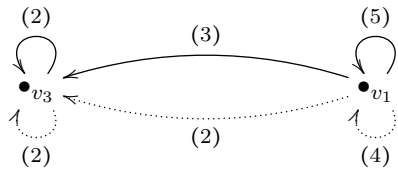
\[
\xymatrix@R=20pt@C=50pt{
& & 
\bullet_{v_3} \ar@{->}_{(2)}@(ur,ul) \ar@{..>}^{(2)}@(dr,dl) & & 
\bullet_{v_1} \ar@{->}_{(5)}@(ur,ul) \ar@{->}_{(3)}@/_1pc/[ll] \ar@{..>}^{(4)}@(dr,dl) \ar@{..>}^{(2)}@/^1pc/[ll]
}
\]
with adjacency matrices
\[
H^{(1)}
=
\begin{bmatrix}
5 & 0 \\
3 & 2
\end{bmatrix}
\qand
H^{(2)}
=
\begin{bmatrix}
4 & 0 \\
2 & 2
\end{bmatrix}.
\]
We see that $\rho(H) = 5$ and $\tau_1$ from above defines an eigenvector of $H^{(1)}$.
So we do not get new traces in $\Tr_{\log 5}^{F}(A)$ in this case.

\smallskip

\noindent
Step 3. Consider the sink subgraph $H$ given by
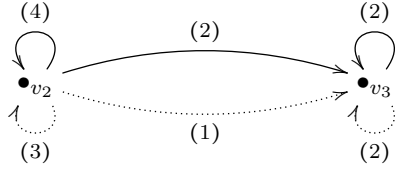
\[
\xymatrix@R=20pt@C=50pt{
\bullet_{v_2} \ar@{->}_{(4)}@(ur,ul) \ar@{->}^{(2)}@/^1pc/[rr] \ar@{..>}^{(3)}@(dr,dl) \ar@{..>}_{(1)}@/_1pc/[rr] & & 
\bullet_{v_3} \ar@{->}_{(2)}@(ur,ul) \ar@{..>}^{(2)}@(dr,dl) & & 
}
\]
with adjacency matrices
\[
H^{(1)}
=
\begin{bmatrix}
4 & 0 \\
2 & 2
\end{bmatrix}
\qand
H^{(2)}
=
\begin{bmatrix}
3 & 0 \\
1 & 2
\end{bmatrix}.
\]
We see that $\rho(H) = 4$ and that $H^{(1)}$ has an eigenvector
\[
\tau_2 := \frac{1}{2} (\de_{v_2} + \de_{v_3})
\]
at $4$.
By checking its support we see that it has $\{2\}$-entropy equal to $\log 3$.
Therefore it contributes one trace in $\Tr_{\log 4}^{\{2\}}(A)$.

\smallskip

\noindent
Step 4. Consider the sink subgraph $H = \overleftarrow{\La}(v_3)$ with adjacency matrices $H^{(1)} = H^{(2)} = \begin{bmatrix} 2 \end{bmatrix}$.
We see that $\rho(H) = 2$ and that $\tau_3 := \de_{v_3}$ is an eigenvector for both $H^{(1)}$ and $H^{(2)}$ at $2$.
Thus it contributes one trace in $\Avt_{\log 2}(A)$.

\smallskip

\noindent
Step 5. We have no other sink subgraphs and now we proceed in computing the entropy of the vertices.
Those will contribute to the $\Tr_\be^{\fty}(A)$.
We see that $\de_{v_1}$ is connected to the irreducible components $[5]$ and $[2]$ of $\La^{(1)}$ and to the irreducible components $[4]$ and $[2]$ of $\La^{(2)}$, and so its entropy is $\log 5$.
Likewise $\de_{v_2}$ and $\de_{v_3}$ have entropy $\log 4$ and $\log 2$, respectively.

\smallskip

Putting all together we see that we have phase transitions at $\log 2, \log 4$ and $\log 5$.
Therefore we have the following cases for $\be \in \bR^+$:

\noindent
$\bullet$ If $\be \in (\log 5, +\infty)$ then
\[
\GEq_\be(\N\T(\La)) = \GEq_\be^{\fty}(\N\T(\La)) \simeq \sca{\de_{v_1}, \de_{v_2}, \de_{v_3}}.
\]
$\bullet$ If $\be = \log 5$ then
\[
\GEq_{\log 5}(\N\T(\La)) = \GEq_{\log 5}^{\fty}(\N\T(\La)) \oplus_{\textup{convex}} \GEq_{\log 5}^{\{2\}}(\N\T(\La))
\]
with
\[
\GEq_{\log 5}^{\fty}(\N\T(\La)) \simeq \sca{\de_{v_2}, \de_{v_3}}
\qand
\GEq_{\log 5}^{\{2\}}(\N\T(\La)) = \{\frac{1}{2}(\de_{v_1} + \de_{v_3}) \}.
\]
$\bullet$ If $\be \in (\log 4, \log 5)$ then
\[
\GEq_\be(\N\T(\La)) = \GEq_\be^{\fty}(\N\T(\La)) \simeq \sca{\de_{v_2}, \de_{v_3}}.
\]
$\bullet$ If $\be = \log 4$ then
\[
\GEq_{\log 4}(\N\T(\La)) = \GEq_{\log 4}^{\fty}(\N\T(\La)) \oplus_{\textup{convex}} \GEq_{\log 4}^{\{2\}}(\N\T(\La))
\]
with
\[
\GEq_{\log 4}^{\fty}(\N\T(\La)) \simeq \{\de_{v_3}\}
\qand
\GEq_{\log 4}^{\{2\}}(\N\T(\La)) = \{\frac{1}{2}(\de_{v_2} + \de_{v_3}) \}.
\]
$\bullet$ If $\be \in (\log 2, \log 4)$ then
\[
\GEq_\be(\N\T(\La)) = \GEq_\be^{\fty}(\N\T(\La)) \simeq \{\de_{v_3}\}.
\]
$\bullet$ If $\be = \log 2$ then
\[
\GEq_{\log 2}(\N\T(\La)) = \GEq_{\log 2}^{\infty}(\N\T(\La)) \simeq \{\de_{v_3} \}.
\]
$\bullet$ If $\be \in (0, \log 2)$ then $\GEq_\be(\N\T(\La)) = \mt$.

\smallskip

Moreover we see that the higher-rank graph has no sources at any, or all colours.
Thus the KMS-structure descends to $\N\O(\La)$ only by keeping the infinite type states.
That is
\[
\GEq_{\log 2}(\N\O(\La)) = \GEq_{\log 2}^{\infty}(\N\O(\La)) \simeq \{\de_{v_3}\},
\]
while $\GEq_\be(\N\O(\La)) = \mt$ for all $\be \neq \log 2$.
\end{example}

\begin{example}
In the next example we remove the cycles at $v_2$ to create a source (but fix the connecting edges between $v_2$ and $v_3$ to have two commuting matrices).
So now let the following rank-2 graph:
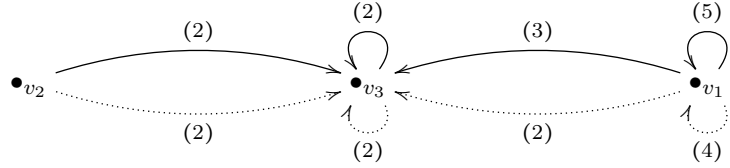
\[
\xymatrix@R=20pt@C=50pt{
\bullet_{v_2} \ar@{->}^{(2)}@/^1pc/[rr] \ar@{..>}_{(2)}@/_1pc/[rr] & & 
\bullet_{v_3} \ar@{->}_{(2)}@(ur,ul) \ar@{..>}^{(2)}@(dr,dl) & & 
\bullet_{v_1} \ar@{->}_{(5)}@(ur,ul) \ar@{->}_{(3)}@/_1pc/[ll] \ar@{..>}^{(4)}@(dr,dl) \ar@{..>}^{(2)}@/^1pc/[ll]
}
\]
We can write the adjacency matrices with respect to the same order $(v_1, v_2, v_3)$ as
\[
\La^{(1)}
=
\begin{bmatrix}
5 & 0 & 0 \\
0 & 0 & 0 \\
3 & 2 & 2
\end{bmatrix}
\qand
\La^{(2)}
=
\begin{bmatrix}
4 & 0 & 0 \\
0 & 0 & 0 \\
2 & 2 & 2
\end{bmatrix}.
\]
The analysis is similar as before with the only difference that in considering the subgraph
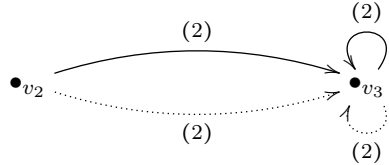
\[
\xymatrix@R=20pt@C=50pt{
\bullet_{v_2} \ar@{->}^{(2)}@/^1pc/[rr] \ar@{..>}_{(2)}@/_1pc/[rr] & & 
\bullet_{v_3} \ar@{->}_{(2)}@(ur,ul) \ar@{..>}^{(2)}@(dr,dl) & & 
}
\]
we just get a common eigenvector $\de_{v_3}$ at $2$.
Also notice that we just have $\log 5$ and $\log 2$ as phase transitions, while the entropy of $v_2$ is now $\log 2$.
Hence for $\be \in \bR^+$ we get:

\noindent
$\bullet$ If $\be \in (\log 5, +\infty)$ then
\[
\GEq_\be(\N\T(\La)) = \GEq_\be^{\fty}(\N\T(\La)) \simeq \sca{\de_{v_1}, \de_{v_2}, \de_{v_3}}.
\]
$\bullet$ If $\be = \log 5$ then
\[
\GEq_{\log 5}(\N\T(\La)) = \GEq_{\log 5}^{\fty}(\N\T(\La)) \oplus_{\textup{convex}} \GEq_{\log 5}^{\{2\}}(\N\T(\La))
\]
with
\[
\GEq_{\log 5}^{\fty}(\N\T(\La)) \simeq \sca{\de_{v_2}, \de_{v_3}}
\qand
\GEq_{\log 5}^{\{2\}}(\N\T(\La)) = \{\frac{1}{2}(\de_{v_1} + \de_{v_3}) \}.
\]
$\bullet$ If $\be \in (\log 2, \log 5)$ then
\[
\GEq_\be(\N\T(\La)) = \GEq_\be^{\fty}(\N\T(\La)) \simeq \sca{\de_{v_2}, \de_{v_3}}.
\]
$\bullet$ If $\be = \log 2$ then
\[
\GEq_{\log 2}(\N\T(\La)) = \GEq_{\log 2}^{\infty}(\N\T(\La)) \simeq \{\de_{v_3} \}.
\]
$\bullet$ If $\be \in (0, \log 2)$ then $\GEq_\be(\N\T(\La)) = \mt$.

In order to get the required simplices for $\N\O(\La)$ we need to check which vertices are $F$-tracing for $F \subseteq \{1, \dots, N\}$, and exclude the traces supported on those.
As only the finite part and the $\{2\}$-part contribute we just need to do so for $F = \{1, 2\}$ and $F = \{2\}$.
For $F = \{1, 2\}$ we see that only $v_2$ is not tracing (being a source).
Hence only $\de_{v_2}$ will pass through from the finite part.
For $F = \{2\}$ we see that $v_2$ is a $\{2\}$-source and $v_3$ is $\{1\}$-connected to that $\{2\}$-source.
So only $v_1$ is $\{2\}$-tracing.
However $\frac{1}{2}(\de_{v_1} + \de_{v_3})$ is not zero on $T_{v_1}$, and so it will not pass through from the $\{2\}$-part.
In particular we lose the phase transition at $\log 5$.
Therefore we have:

\noindent
$\bullet$ If $\be \in (\log 2, +\infty)$ then
\[
\GEq_\be(\N\O(\La)) = \GEq_\be^{\fty}(\N\O(\La)) \simeq \{\de_{v_2}\}.
\]
The induced finite-type KMS-state can be constructed by statistical approximations on the Cunt-Krieger representation induced by the paths starting at $v_2$.

\noindent
$\bullet$ If $\be = \log 2$ then
\[
\GEq_{\log 2}(\N\O(\La)) = \GEq_{\log 2}^{\infty}(\N\O(\La)) \simeq \{\de_{v_3} \}.
\]
$\bullet$ If $\be \in (0, \log 2)$ then $\GEq_\be(\N\O(\La)) = \mt$.
\end{example}

\begin{example}
More generally, Kumjian-Pask \cite{KP00} have shown that the following two-coloured graph
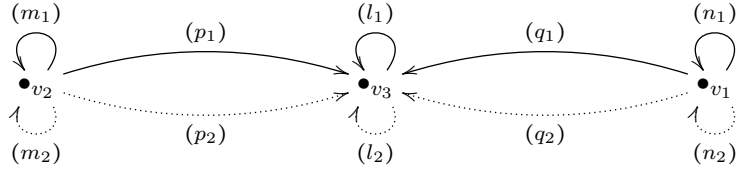
\begin{equation*}
\xymatrix@R=20pt@C=50pt{
\bullet_{v_2} \ar@{->}_{(m_1)}@(ur,ul) \ar@{->}^{(p_1)}@/^1pc/[rr] \ar@{..>}^{(m_2)}@(dr,dl) \ar@{..>}_{(p_2)}@/_1pc/[rr] & & 
\bullet_{v_3} \ar@{->}_{(l_1)}@(ur,ul) \ar@{..>}^{(l_2)}@(dr,dl) & & 
\bullet_{v_1} \ar@{->}_{(n_1)}@(ur,ul) \ar@{->}_{(q_1)}@/_1pc/[ll] \ar@{..>}^{(n_2)}@(dr,dl) \ar@{..>}^{(q_2)}@/^1pc/[ll]
}
\end{equation*}
defines a rank-2 structure if and only if the adjacency matrices
\[
\La^{(1)}
=
\begin{bmatrix}
n_1 & 0 & 0 \\
0 & m_1 & 0 \\
q_1 & p_1 & l_1
\end{bmatrix}
\qand
\La^{(2)}
=
\begin{bmatrix}
n_2 & 0 & 0 \\
0 & m_2 & 0 \\
q_2 & p_2 & l_2
\end{bmatrix}
\]
commute.
Let us now describe the phase transitions for the possible values that give such a rank-2 graph.
We begin with the following remark.
Let the matrix
\[
A =
\begin{bmatrix}
n & 0 & 0 \\
0 & m & 0 \\
p & q & l
\end{bmatrix}.
\]
If $n \geq m > l$ then $A$ has nonnegative unimodular eigenvectors
\[
w_n = \frac{n-l}{n-l+p} \begin{bmatrix} 1 \\ 0 \\ p/(n-l) \end{bmatrix}, \;
w_m = \frac{m-l}{m-l+q} \begin{bmatrix} 0 \\ 1 \\ q/(m-l) \end{bmatrix}, \;
w_l = \begin{bmatrix} 0 \\ 0 \\ 1 \end{bmatrix},
\]
at $n,m,l$ respectively.
If $n > l \geq m$ then $A$ has nonnegative unimodular eigenvectors
\[
w_n = \frac{n-l}{n-l+p} \begin{bmatrix} 1 \\ 0 \\ p/(n-l) \end{bmatrix}
\qand
w_l = \begin{bmatrix} 0 \\ 0 \\ 1 \end{bmatrix},
\]
at $n, l$ respectively.
Finally if $l > m, n$ then $A$ has $w_l$ as a single nonnegative unimodular eigenvector at $l$.
Symmetry over $n,m$ gives the other cases as well.
Now let us return to the rank-2 graph and see how the possible order of $n,m,l$ affects the gauge-invariant KMS-simplices for
\[
n:= \max\{n_1, n_2\}, \; m:= \max\{m_1, m_2\}, \; l := \max\{l_1, l_2\}.
\]

We consider the case where $n_i, m_i, l_i >0$.
We leave it as an exercise to the reader to check what happens when some of them are zero.

\smallskip

\noindent
{\bf Case 1.} Suppose that $n > m > l$.
Then we have three phase transitions at $\log l$, $\log m$ and $\log n$.
The finite-type simplex for $\La$ becomes
\[
\GEq_\be^{\fty}(\N\T(\La))
\simeq
\begin{cases}
\sca{\de_{v_1}, \de_{v_2}, \de_{v_3}} & \text{ if } \be \in (\log n, +\infty), \\[5pt]
\sca{\de_{v_2}, \de_{v_3}} & \text{ if } \be \in (\log m, \log n], \\[5pt]
\{\de_{v_3}\} & \text{ if } \be \in (\log l, \log m].
\end{cases}
\]
We have a single subharmonic part at each $\log n, \log m, \log l$ induced by the unimodular eigenvectors $w_n, w_m ,w_l$.
The type depends on how we achieve the maximums $n, m, l$, i.e.,
\[
\{w_n\} 
\simeq
\begin{cases}
\GEq_{\log n}^{\{1\}}(\N\T(\La)) & \text{ if } n_1 < n_2, \\[8pt]
\GEq_{\log n}^{\{2\}}(\N\T(\La)) & \text{ if } n_1 > n_2, \\[8pt]
\GEq_{\log n}^{\infty}(\N\T(\La)) & \text{ if } n_1 = n_2.
\end{cases}
\] 
Likewise for $w_m$ and $w_l$ at $\log m$ and $\log l$.

\smallskip

\noindent
{\bf Case 2.} Suppose that $n = m > l$.
Then we have two phase transitions at $\log n$ and $\log l$.
The finite-type simplex for $\La$ becomes
\[
\GEq_\be^{\fty}(\N\T(\La))
\simeq
\begin{cases}
\sca{\de_{v_1}, \de_{v_2}, \de_{v_3}} & \text{ if } \be \in (\log n, +\infty), \\[5pt]
\{\de_{v_3}\} & \text{ if } \be \in (\log l, \log n].
\end{cases}
\]
The $w_n$ and $w_m$ each induce a subharmonic part whose type depends on how $n,m$ are achieved, as in Case 1.
For example we have
\[
\GEq_{\log n}^{\infty}(\N\T(\La)) \simeq \sca{w_n, w_m} \qif n_1 = n_2, m_1 = m_2.
\]
On the other hand $w_l$ defines a single subharmonic part at $\log l$.

\smallskip

\noindent
{\bf Case 3.} Suppose that $n > l \geq m$.
Then we have two phase transitions at $\log n$ and $\log l$.
The finite simplex for $\La$ becomes
\[
\GEq_\be^{\fty}(\N\T(\La))
\simeq
\begin{cases}
\sca{\de_{v_1}, \de_{v_2}, \de_{v_3}} & \text{ if } \be \in (\log n, +\infty), \\[5pt]
\sca{\de_{v_2}, \de_{v_3}} & \text{ if } \be \in (\log l, \log n]. 
\end{cases}
\]
Now $w_n$ and $w_l$ each induce a single subharmonic part at $\log n$ and $\log l$ whose type depends on how we achieve the maximums, as in Case 1.

\smallskip

\noindent
{\bf Case 4.} Suppose that $l \geq n, m$.
Then we have a single phase transition at $\log l$ so that
\[
\Eq_\be(\N\T(\La)) = \Eq_\be^{\fty}(\N\T(\La)) \simeq \sca{\de_{v_1}, \de_{v_2}, \de_{v_3}}
\foral \be \in (\log l, +\infty),
\]
while $w_l$ induces a single KMS-state at $\log l$ in the sense that
\[
\{\de_{v_3}\} 
\simeq
\GEq_{\log l}(\N\T(\La))
=
\begin{cases}
\GEq_{\log l}^{\{1\}}(\N\T(\La)) & \text{ if } l_1 < l_2, \\[8pt]
\GEq_{\log l}^{\{2\}}(\N\T(\La)) & \text{ if } l_1 > l_2, \\[8pt]
\GEq_{\log l}^{\infty}(\N\T(\La)) & \text{ if } l_1 = l_2.
\end{cases}
\] 

\smallskip

We have no sources of any coulour, and so $\N\O(\La)$ admits the KMS-states that are just of infinite type, when those exist.
\end{example}

The finite-type KMS-states of the next example has been worked out also by an Huef-Raeburn in \cite{HR19} (which appeared on the arXiv while the current paper was in submission).
Apart from \cite{Kak18}, the structure of the finite-type states has been verified also by Christensen \cite{Chr18} for higher-rank graphs, and by Afsar-Larsen-Neshveyev \cite{ALN18} for rather general product systems.

\begin{example}
Let the following 2-coloured graph:
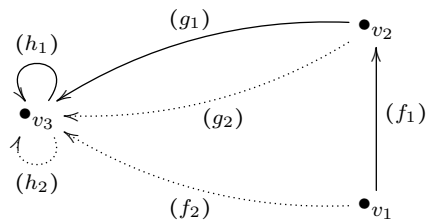
\[
\xymatrix@R=20pt@C=50pt{
& & \bullet_{v_2} \ar@{->}_{(g_1)}@/_1pc/[lld] \ar@{..>}^{(g_2)}@/^1pc/[lld] \\
\bullet_{v_3} \ar@{->}_{(h_1)}@(ur,ul) \ar@{..>}^{(h_2)}@(dr,dl) & & \\
& & \bullet_{v_1} \ar@{..>}^{(f_2)}@/^1pc/[llu] \ar@{->}_{(f_1)}[uu]
}
\]
with adjacency matrices
\[
\La^{(1)}
=
\begin{bmatrix}
0 & 0 & 0 \\
f_1 & 0 & 0 \\
0 & g_1 & h_1
\end{bmatrix}
\qand
\La^{(2)}
=
\begin{bmatrix}
0 & 0 & 0 \\
0 & 0 & 0 \\
f_2 & g_2 & h_2
\end{bmatrix}.
\]
The 2-coloured graph induces the skeleton of a rank-2 graph if and only if the matrices commute, as in this case we can induce a pairing between two-coloured paths, i.e., if and only if $h_1 f_2 = g_2 f_1$ and $h_1 g_2 = h_2 g_1$.
In this case we see that we have one phase transition at 
\[
\log h := \max\{ \log h_1, \log h_2\}.
\]
Therefore we have that
\[
\Eq_\be(\N\T(\La)) = \Eq_\be^{\fty}(\N\T(\La)) \simeq \sca{\de_{v_1}, \de_{v_2}, \de_{v_3}}
\foral \be \in (\log h, +\infty).
\]
On the other hand $\de_{v_3}$ is the common eigenvector and induces a single subharmonic part at $\log h$.
Since its $\{i\}$-entropy equals $\log h_i$ the type each time depends on the relation between $h_1$ and $h_2$, i.e.,
\[
\{\de_{v_3}\} 
\simeq 
\GEq_{\log h}(\N\T(\La))
=
\begin{cases}
\GEq_{\log h}^{\{1\}}(\N\T(\La)) & \text{ if } h_1 < h_2, \\[8pt]
\GEq_{\log h}^{\{2\}}(\N\T(\La)) & \text{ if } h_1 > h_2, \\[8pt]
\GEq_{\log h}^{\infty}(\N\T(\La)) & \text{ if } h_1 = h_2.
\end{cases}
\]
We have one source at $v_1$ and therefore the finite simplex descends to $\N\O(\La)$ as
\[
\Eq_\be(\N\O(\La)) = \Eq_\be^{\fty}(\N\O(\La)) \simeq \{\de_{v_1}\}
\foral 
\be \in (\log h, +\infty).
\]
Alternatively, the induced KMS-state can be constructed by statistical approximations on the Cuntz-Krieger representation given by the paths starting at $v_1$.
On the other hand $v_3$ is $\{i\}$-communicated by the source $v_1$ and so it is not $F$-tracing for all $F = \{i\}$.
Hence the gauge-invariant KMS-simplex at $\log h$ descends as is to $\N\O(\La)$.
\end{example}

\begin{example}
Let the following 2-coloured graph:
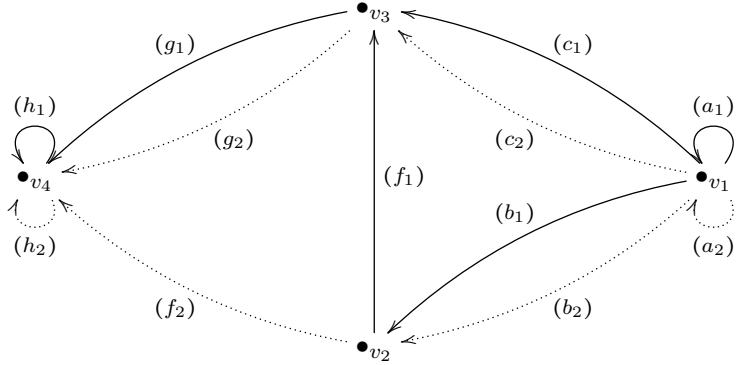
\[
\xymatrix@R=50pt@C=50pt{
& & \bullet_{v_3} \ar@{->}_{(g_1)}@/_1pc/[lld] \ar@{..>}^{(g_2)}@/^1pc/[lld] & & \\
\bullet_{v_4} \ar@{->}_{(h_1)}@(ur,ul) \ar@{..>}^{(h_2)}@(dr,dl) & & & & \bullet_{v_1} \ar@{->}_{(a_1)}@(ur,ul) \ar@{..>}^{(a_2)}@(dr,dl) \ar@{->}_{(c_1)}@/_1pc/[llu] \ar@{..>}^{(c_2)}@/^1pc/[llu]  \ar@{->}_{(b_1)}@/_1pc/[lld] \ar@{..>}^{(b_2)}@/^1pc/[lld] \\
& & \bullet_{v_2} \ar@{..>}^{(f_2)}@/^1pc/[llu] \ar@{->}_{(f_1)}[uu] & &
}
\]
with adjacency matrices
\[
\La^{(1)}
=
\begin{bmatrix}
a_1 & 0 & 0 & 0 \\
b_1 & 0 & 0 & 0 \\
c_1 & f_1 & 0 & 0 \\
0 & 0 & g_1 & h_1
\end{bmatrix}
\qand
\La^{(2)}
=
\begin{bmatrix}
a_2 & 0 & 0 & 0 \\
b_2 & 0 & 0 & 0 \\
c_2 & 0 & 0 & 0 \\
0 & f_2 & g_2 & h_2
\end{bmatrix}.
\]
Once more the 2-coloured graph induces the skeleton of a rank-2 graph if and only if the matrices commute.
Here we see that we have possible phase transitions at 
\[
\log a := \max\{ \log a_1, \log a_2 \}
\qand
\log h := \max\{ \log h_1, \log h_2 \}.
\]
We consider the following two cases for $a_i, h_i >0$.
Note that in both cases the trace $\de_{v_4}$ induces a subharmonic part.

\smallskip

\noindent
{\bf Case 1.} Suppose that $a > h$.
By checking the forward communicating components for the vertices we have that the entropy of $v_1$ is $\log a$ and the entropies of the rest is $\log h$.
Thus
\[
\GEq_{\be}^{\fty}(\N\T(\La))
\simeq
\begin{cases}
\sca{\de_{v_1}, \de_{v_2}, \de_{v_3}, \de_{v_4}} & \text{ if } \be \in (\log a, +\infty), \\[8pt]
\sca{\de_{v_2}, \de_{v_3}, \de_{v_4}} & \text{ if } \be \in (\log h, \log a], \\[8pt]
\mt & \text{ if } \be < \log h.
\end{cases}
\]
If $a_i \neq h_i$ then we get the nonnegative eigenvectors for $\La^{(1)}$ and $\La^{(2)}$ at $a_1$ and $a_2$, respectively, by:
\[
w_1 :=
\begin{bmatrix}
1 \\ b_1/a_1 \\ (c_1 + \frac{b_1 f_1}{a_1})/a_1 \\ g_1(c_1 + \frac{b_1 f_1}{a_1})/a_1 (a_1 - h_1)
\end{bmatrix}
\qand
w_2 :=
\begin{bmatrix}
1 \\ b_2/a_2 \\ c_2/a_2 \\ (b_2 f_2 + c_2 g_2)/a_2 (a_2 - h_2)
\end{bmatrix}.
\]
If $a_1 = a_2$ then there is a unique averaging trace $\tau$ induced by the unimodular forms of $w_1$ and $w_2$.
Indeed commutativity gives that $\La^{(2)} w_1$ is in the $\La^{(1)}$-eigencone of $w_1$, and since the first entry of $w_1$ is $1$ we get that $\La^{(2)} w_1 = a_2 w_1$.
As the eigencone of $\La^{(2)}$ at $a_2$ is generated by a unique vector we get that $w_1$ is a multiple of $w_2$, and thus they have the same unimodular form.
If $a_1 > a_2$ then the trace $\tau_1$ induced by $w_1$ will have $\{2\}$-entropy equal to $\log a_2 < \log a$, and likewise if $a_1 < a_2$.
Therefore we have for the non-finite part that
\begin{align*}
\GEq_{\log a}(\N\T(\La)) \setminus \GEq_{\log a}^{\fty}(\N\T(\La)) 
& =
\begin{cases}
\GEq_{\log a}^{\{1\}}(\N\T(\La)) & \text{ if } a_1 < a_2, \\[8pt]
\GEq_{\log a}^{\{2\}}(\N\T(\La)) & \text{ if } a_1 > a_2, \\[8pt]
\GEq_{\log a}^{\infty}(\N\T(\La)) & \text{ if } a_1 = a_2,
\end{cases}
\\
& \simeq
\begin{cases}
\{\tau_2\} & \text{ if } a_1 < a_2, \\[8pt]
\{\tau_1\} & \text{ if } a_1 > a_2, \\[8pt]
\{\tau\} & \text{ if } a_1 = a_2.
\end{cases}
\end{align*}
Moreover we have that $\de_{v_4}$ defines a subharmonic part at $\log h$ and in particular
\[
\{\de_{v_4}\} 
\simeq 
\GEq_{\log h}(\N\T(\La))
=
\begin{cases}
\GEq_{\log h}^{\{1\}}(\N\T(\La)) & \text{ if } h_1 < h_2, \\[8pt]
\GEq_{\log h}^{\{2\}}(\N\T(\La)) & \text{ if } h_1 > h_2, \\[8pt]
\GEq_{\log h}^{\infty}(\N\T(\La)) & \text{ if } h_1 = h_2.
\end{cases}
\]

\smallskip

\noindent
{\bf Case 2.} Suppose that $a \leq h$.
Then we have one phase transition at $\log h$.
Now all vertices have the same entropy $\log h$ and so
\[
\GEq_{\be}^{\fty}(\N\T(\La))
\simeq
\begin{cases}
\sca{\de_{v_1}, \de_{v_2}, \de_{v_3}, \de_{v_4}} & \text{ if } \be \in (\log h, +\infty), \\[8pt]
\mt & \text{ if } \be < \log h.
\end{cases}
\]
We have nine possible cases based on the order of $a_1$ with $a_2$, and of $h_1$ with $h_2$.
In all cases we get that either $a_1 \leq h_1$ or $a_2 \leq h_2$, and the only common nonnegative eigenvector for $\La^{(1)}$ and $\La^{(2)}$ at $\log h$ is then $\de_{v_4}$.
Hence we obtain as before:
\[
\{\de_{v_4}\} 
\simeq 
\GEq_{\log h}(\N\T(\La))
=
\begin{cases}
\GEq_{\log h}^{\{1\}}(\N\T(\La)) & \text{ if } h_1 < h_2, \\[8pt]
\GEq_{\log h}^{\{2\}}(\N\T(\La)) & \text{ if } h_1 > h_2, \\[8pt]
\GEq_{\log h}^{\infty}(\N\T(\La)) & \text{ if } h_1 = h_2.
\end{cases}
\]

\smallskip

In all cases we see that all vertices are $F$-tracing and so only the infinite-type gauge-invariant KMS-simplex descends to $\N\O(\La)$ (whenever it exists).
\end{example}

\section{Weighted dynamics}

As noticed $\N\O(\La)$ can have an empty KMS-simplex, and this may not be desirable.
This can be the case for example when $\La$ has no $F$-sources so that the KMS-simplex consists only of infinite-type states, i.e., of common eigenvectors at the same eigenvalue; and there may be none.
A way around is to apply a scaling on the rotational action.
Indeed, every commuting family $(\La^{(1)}, \dots, \La^{(N)})$ has a common eigenvector $\tau$, i.e., there are $\la_i \geq 0$ such that $\La^{(i)} \tau = \la_i \tau$.
If all $\la_i \neq 0$ then we will see that $\tau$ induces a KMS-state of infinite type at $\be = 1$ for the positively weighted rotational action
\[
\bR \ni r \mapsto \ga_{(e^{ir \la_1}, \dots, e^{ir \la_N})} \in \Aut(\N\T(\La)).
\]
In hindsight one uses appropriate weights to move the numerical radii to a common number so that the $F$-subharmonic parts integrate into the infinite-type simplex for the new action.
Of course one can use different weights to achieve common eigenvalues at the sink subgraphs.
When every $\La^{(i)}$ is irreducible we weight by $\rho(\La^{(i)})$.

By tweaking some bits, our analysis accommodates this setting and allows to compute the KMS-simplices for all (positively weighted) dynamics.
To this end let $s_1, \dots, s_N > 0$ and for every $r \in \bR$ define the action
\[
\si_r' = \ga_{(e^{irs_1}, \dots, e^{ir s_N})} \in \Aut(\N\T(\La)).
\]
In particular we have that
\[
\si_r(T_{\un{\la}} T_{\umu}^*) = e^{i \sca{\ell(\un{\la}) - \ell(\umu), \un{s}} r} T_{\un{\la}} T_{\umu}^*
\qfor
\sca{\ell(\un{\la}) - \ell(\umu), \un{s}} := \sum_{i=1}^N (|\la_i| - |\mu_i|) s_i.
\]
Note that if $s_1 = \cdots = s_N = 1$ then we obtain
\[
\sca{\ell(\un{\la}) - \ell(\umu), \un{s}} = |\un{\la}| - |\umu|.
\]
We may thus proceed in the same way as in \cite{Kak18} (see also \cite{Kak17b}) but now we substitute every occurrence of $|\un{n}| \be$ with $\sca{\un{n}, \un{s}} \be$.
As a consequence the $\{i\}$-strong entropy of $\La$ with respect to $\si'$ is $h_\La^{s, \{i\}}/s_i$, in the sense that the series 
\[
\sum_{k \in \bZ_+} \| \sum_{\mu_i \in \La^{(i)}, |\mu_i| = k} e^{- \sca{\mu_i, \un{s}} \be} T_{\mu_i}^* T_{\mu_i} \|
=
\sum_{k \in \bZ_+} e^{-k s_i \be} \| \sum_{\mu_i \in \La^{(i)}, |\mu_i| = k} T_{\mu_i}^* T_{\mu_i} \|
\]
converges when $\be > h_{\La}^{s, \{i\}}/s_i$.
Therefore the role of the $\si'$-strong entropy is now played by
\[
\si'\text{-}h_\La^s = \max\{ h_\La^{s, \{i\}}/s_i \mid i=1, \dots, N\}.
\]
This much is true for any product system of finite rank.
Now for higher-rank graphs in particular we have that the tracial entropy is described by the numerical radii of the components it connects with.
Also notice that if $\vphi$ is a $(\si',\be)$-KMS state then
\[
\vphi(T_v Q_{\Bi})
= \vphi(T_v) - \sum_{\ell(e) = \Bi} \vphi(T_v T_{e} T_{e}^*)
= \vphi(T_v) - e^{-s_i \be} \sum_{\ell(e) = \Bi} \vphi(T_{e}^* T_v T_{e}).
\]
Hence a $\tau \in \Tr(A)$ induces an $F$-subharmonic $(\si',\be)$-KMS state if and only if
\[
\sum \big\{ e^{-\sca{\umu, \un{s}}\be} \tau(T_{\umu}^* T_{\umu}) \mid \ell(\umu) \in F \big\} < \infty
\qand
\sum_{\ell(e) = \Bi} \tau(T_{e}^* \cdot T_{e}) = e^{s_i \be} \tau(\cdot) \foral i \notin F.
\]
By proceeding as in Proposition \ref{P:F chara}, we get that if
\[
\overleftarrow{\La}(\supp \tau, F) := (H^{(1)}, \dots, H^{(N)}; \sim)
\qand
\rho(H^{(i_0)})/s_{i_0}
:=
\max\{ \rho(H^{(i)})/s_i \mid i \in F\},
\]
then $\tau$ induces an $F$-subharmonic $(\si',\be)$-KMS state if and only if
\[
\rho(H^{(i_0)}) < s_{i_0} \be
\qand
\La^{(i)} \tau = e^{s_i \be} \tau \foral i \notin F.
\]
Accordingly to Theorem \ref{T:main gi} the positive transitions occur at $\max_i \rho(H_i)/s_i$ for all sink subgraphs $H = (H^{(1)}, \dots, H^{(N)}; \sim)$ of $\La$.


\end{document}